\documentclass[a4paper,12pt]{amsart}
\usepackage[latin1]{inputenc}
\usepackage{amsmath, amsthm, amssymb, amsfonts, amscd}
\usepackage[english]{babel}
\usepackage[all]{xy}
\usepackage[linktocpage=true]{hyperref}
\usepackage{url}
\usepackage[usenames, dvipsnames]{color}
\usepackage{faktor}
\usepackage{pinlabel}
\usepackage{tikz-cd}\newtheorem{thm}{Theorem}[section]
\usepackage[margin=3cm, marginpar=2.5cm]{geometry}

\newtheorem{lemma}[thm]{Lemma}
\newtheorem{prop}[thm]{Proposition}
\newtheorem{cor}[thm]{Corollary}

\theoremstyle{remark}
\newtheorem{defi}[thm]{Definition}

\newcommand{\Z}{\mathbb{Z}}
\newcommand{\Q}{\mathbb{Q}}
\newcommand{\R}{\mathbb{R}}
\newcommand{\C}{\mathcal{C}}

\newcommand{\F}{\mathbb{F}}

\newcommand{\D}{D}

\newcommand{\inv}{^{-1}}

\newcommand{\es}{S^2 \times S^1}

\newcommand{\knoti}[1]{\kappa(#1)}

\newcommand{\dtw}{\underline{d}}
\newcommand{\de}{\partial}

\newcommand{\Yt}{(Y,\ft)}
\newcommand{\ft}{\mathfrak{t}}
\newcommand{\fs}{\mathfrak{s}}
\newcommand{\fto}{\ft^{\rm op}}

\newcommand{\spinc}{spin$^c$~}
\newcommand{\Spinc}{{\rm Spin}^c}
\newcommand{\Zmod}[1]{\Z/{#1}\Z}
\newcommand{\gsh}{g^0_{\rm sh}}

\DeclareMathOperator{\geomi}{\pitchfork}

\DeclareMathOperator{\PD}{PD}
\DeclareMathOperator{\HF}{HF^+}
\DeclareMathOperator{\HFo}{HF^+_{odd}}
\DeclareMathOperator{\HFe}{HF^+_{even}}
\DeclareMathOperator{\HFoo}{HF^\infty}
\DeclareMathOperator{\tHF}{\underline{HF}^+}
\DeclareMathOperator{\HFred}{HF_{red}}
\DeclareMathOperator{\HFL}{\widehat{HFL}}

\DeclareMathOperator{\lk}{lk}

\makeatletter
\let\@wraptoccontribs\wraptoccontribs
\makeatother

\title[Heegaard Floer homology, concordance, and the Thurston norm]{Heegaard Floer homology and concordance bounds on the Thurston norm}

\author{Daniele Celoria}
\address{Mathematical Institute, University of Oxford, Oxford, UK}
\email{Daniele.Celoria@maths.ox.ac.uk}

\author{Marco Golla}
\address{CNRS, Laboratoire de Math\'ematiques Jean Leray, Nantes, France}
\email{marco.golla@univ-nantes.fr}

\contrib[with an appendix with]{Adam Simon Levine}


\date{}
\begin{document}

\begin{abstract}
We prove that twisted correction terms in Heegaard Floer homology provide lower bounds on the Thurston norm of certain cohomology classes determined by the strong concordance class of a $2$-component link $L$ in $S^3$. We then specialise this procedure to knots in $\es$, and obtain a lower bound on their geometric winding number. We then provide an infinite family of null-homologous knots with increasing geometric winding number, on which the bound is sharp.
\end{abstract}

\maketitle

\section{Introduction}

Consider a 2-component link $L = K_0 \cup K_1 \subset S^3$, such that $\lk(K_0,K_1) = 0$. Recall that two such links are \emph{strongly concordant} if they are the boundary of a pair of disjoint properly embedded smooth annuli in $S^3 \times [0,1]$.

In this note we are going to show that twisted correction terms, defined by Behrens and the second author in~\cite{behrens2015heegaard}, can be used to give lower bounds on the Thurston norm~\cite{thurston1986norm} $x$ of certain cohomology classes, determined by the strong concordance class of the link $L$. More specifically, call $\mu_i$ the meridian of the component $K_i$; one can consider the minimum attained by $x$ on the classes $\PD[\mu_i^\prime] \in H^2(S^3,L^\prime)$ over all links $L^\prime = K_0^\prime \cup K_1^\prime$ strongly concordant to $L$. As a notational shortcut we will always assume that the each connected component of the concordance cobounds $K_i \cup -K_i^\prime$.
 
\begin{thm}\label{t:mainthurston}
Let $L = K_0 \cup K_1$ be a $2$-component link, with $\lk(K_0,K_1) =0$. Let $Y$ be the $3$-manifold obtained as $0$-surgery along $K_0$ and $1$-surgery along $K_1$. Then
\begin{equation}\label{e:geomlowerboundthurstonconcordance}
\min_{L^\prime \sim L } \left\lceil\frac{x(\PD[\mu_0^\prime]) +1}4 \right\rceil \ge \frac{\dtw(Y) + \dtw(-Y) + 1}2.
\end{equation}
\end{thm}

Here $\dtw(Y)$ denotes the correction term of $\tHF (Y)$, the Heegaard Floer homology with fully twisted coefficients, in the unique $\ft \in \Spinc (Y)$ with vanishing Chern class. We are actually going to prove a slightly stronger result (Theorem~\ref{t:geomlowerboundthurston}) in Section~\ref{sec:dimostrazione}. 

In what follows we specialise Theorem~\ref{t:mainthurston} to $2$-component links with one trivial component, on which we perform a $0$-framed surgery. Note that, by the positive solution to the Property R conjecture~\cite{gabai1987foliations} this is the only possible case in which the image of the other component becomes a knot in $\es$ after the surgery.

In their seminal work~\cite{ozsvath2004holomorphicknot} Ozsv\'ath and Szab\'o define knot Floer homology for (null-homologous) links in a general $3$-manifold, by a process they call \emph{knotification};
this procedure associates to a $n$-component null-homologous link $L$ in the $3$-manifold $Y$, a null-homologous knot in $Y \#^{n-1} S^2 \times S^1$. 

Recently, this construction has been exploited by Hedden and Kuzbary~\cite{heddenkuzbary} to provide a further way of defining a concordance group of links in $S^3$ (see also~\cite{hosokawa1967concept} and~\cite{donald2012concordance} for previous approaches to the definition of such a group).

Now consider a link $L = \bigcirc \cup K \subset S^3$, with $\lk(\bigcirc,K) =0$; by doing $0$-surgery on $\bigcirc$, the other component becomes a knot in $\es$. Following~\cite{davis2017concordance} we define the \emph{geometric winding number} $\geomi(K)$; this is just the minimal geometric intersection number between a knot $K \subset \es$ and a $2$-sphere generating $H_2(\es;\Z)$; this invariant was called \emph{wrapping number} in~\cite{wrappinglivingston}. 
We can state the bound given by Theorem~\ref{t:mainthurston} in this case, and obtain an obstruction to being knotified, up to concordance in $S^2 \times S^1$.

\begin{thm}\label{t:main}
Let $K$ be a null-homologous knot in $\es$, and let $Y_K$ the $3$-manifold obtained as $+1$-surgery along $K$. Then
\begin{equation}\label{e:geomlowerbound}
\min_{K^\prime \sim K} \left\lceil\frac{\geomi(K^\prime)}4\right\rceil \ge \frac{\dtw(Y_K) + \dtw(-Y_K) + 1}2.
\end{equation}
\end{thm}
Note that if the right-hand side of Equation~\eqref{e:geomlowerbound} is greater than $2$, then the knot is not concordant to the knotification of a $2$-component link.

In the case of essential knots in $\es$, we will obtain a similar bound (Theorem~\ref{t:geomi-essential}) on the geometric winding number, building on earlier work by Levine, Ruberman and Strle~\cite{levine2015nonorientable}.
An analogous result, expressed explicitly as a bound on the non-orientable Thurston norm, was discovered by Ni and Wu~\cite{ni2015correction}.

Previous bounds on the geometric winding number can be given in the topological locally flat category using an invariant defined by Schneiderman~\cite{schneiderman2003algebraic};
in fact, this observation was not made explicitly in his paper, so we provide a short proof below.
Then we compare our bound to his, and provide an infinite family of knots in $\es$ which are not concordant to knotified $2$-component links, and where the bound~\eqref{e:geomlowerbound} is stronger.

We note that one should not expect to obtain inequalities analogous to~\eqref{e:geomlowerboundthurstonconcordance} and \eqref{e:geomlowerbound} by means of the ordinary correction terms for rational homology spheres, at least not using arguments akin to the one we exploit in this paper.
Indeed, it appears that it is the lack of symmetry under orientation-reversal of $\dtw$ (or, indeed, of the bottom-most correction term $d_b$) that allows for the right-hand side of both inequalities to be non-trivial.
For ordinary correction terms, an analogue of the right-hand side of~\eqref{e:geomlowerboundthurstonconcordance} would be a multiple of $d(Z) + d(-Z)$, which vanishes since $d(Z) = -d(Z)$ for any integral homology sphere.
Twisted correction terms seem also hard to replace with the more classical bottom-most correction terms, as surfaces with trivial normal bundle naturally appear in this context, and their boundaries do not have standard $\HFoo$ (and in particular bottom-most correction terms are not defined).
This lack of symmetry was already exploited in a different context by Levine and Ruberman~\cite{LevineRuberman}.

Finally, in the appendix (with Adam Levine), we show that correction terms give a lower bound on the 0-shake-slice genus of a knot.
To this end, let $X_K$ denote the trace of the 0-surgery along $K$, i.e. the 4-manifold obtained from $B^4$ by gluing a 0-framed 2-handle along $K$.

\begin{defi}
The $0$-shake-slice genus $\gsh(K)$ of $K$ is the minimum of $g(F)$ as $F$ varies among all smoothly embedded surfaces representing a generator of $H_2(X_K)$.
\end{defi}

Piccirillo recently proved that the 0-shake-slice genus can be strictly less than the 4-ball genus~\cite{piccirillo}, which resolves a problem of Kirby~\cite[Problem 1.41]{kirbylist}.
Moreover, it was previously known that $d(S^3_1(K))$ bounds the 4-ball genus~\cite{Rasmussen-GodaTeragaito};
our theorem shows that correction terms bound the 0-shake-slice genus as well. 

\begin{thm}\label{t:shake-bound}
For any knot $K\subset S^3$, we have $\gsh(K) \ge \dtw(S^3_0(K)) - \frac12$.
\end{thm}

In fact, as we shall see in Proposition~\ref{p:0surgery} below, $\dtw(S^3_0(K)) - \frac12 = d(S^3_{-1}(K)) - 1$.

\subsection*{Organisation of the paper}
The paper is structured as follows.
We will give a few preliminary topological definitions in Section~\ref{sec:knotification}, while the relevant notions on Heegaard Floer homology with twisted coefficients will be recalled in Section~\ref{sec:hf}.
We will then give the proof of Theorem~\ref{t:geomlowerboundthurston} in Section~\ref{sec:dimostrazione}, and deduce Theorems~\ref{t:mainthurston} and~\ref{t:main} from it.
We prove a version of the bound~\eqref{e:geomlowerbound} for essential knots in Section~\ref{sec:essential}.
We explicitly compute the obstruction on a family of examples in Section~\ref{sec:esempi}, and compare our bounds to a bound derived from an invariant of topological concordance due to Schneiderman.
Finally, in the appendix we prove Theorem~\ref{t:shake-bound}.

\subsection*{Notation}
Unless explicitly stated, all manifolds will be smooth and oriented, and all submanifolds will be smoothly embedded;
(singular) homology and cohomology will be taken with integer coefficients.
In the Heegaard Floer context, we will work over the field $\F = \Z/2\Z$ with two elements.
The unknot will be denoted by $\bigcirc$, and if $L$ is a link, $\nu(L)$ will denote a regular tubular neighbourhood of $L$.

\subsection*{Acknowledgments}
We would like to thank Agnese Barbensi, Andr\'as Juh\'asz, Miriam Kuzbary, and Marc Lackenby for their comments and support. A special thanks to Marco Marengon, for his constant support and interest in the project.
We also thank Mark Powell and JungHwan Park for referring us to Schneiderman's work and sharing their expertise and ideas. We want to thank the authors of~\cite{friedl2016satellites} for pointing out a wrong citation in the first version of this paper.
Theorem~\ref{t:shake-bound} arose from a conversation during a conference at the Max Planck Institute in October 2016; we are grateful to MPI for fostering this collaboration.
The authors acknowledge support from the European Research Council (ERC) under the European Unions Horizon 2020 research and innovation programme (grant agreement No 674978). 
Finally, we thank the referee for their useful comments.

\section{Preliminaries}\label{sec:knotification}

Two knots $K_0,K_1$ in a closed, orientable and connected $3$-manifold $Y$ are said to be \emph{concordant} if there exists a smooth and properly embedded cylinder $A \cong S^1 \times [0,1] \subset Y \times [0,1]$, transverse to the boundary, and such that $A \cap (Y\times \{i\}) = K_i$ for $i =0,1$.
Concordance is an equivalence relation on the set of knots in $Y$; the equivalence classes under this relation can be given a group structure if $Y = S^3$, with the operation induced by connected sum.
Also, in this special case, we can equivalently say that two knots $K_0, K_1 \subset S^3$ are concordant if and only if $K_0 \# -K_1$ is \emph{slice}, i.e.~it is the boundary of a smooth and properly embedded disk in $\D^4$.
Here by $-K_1$ we mean the mirror of $K_1$, with its orientation reversed.
Correspondingly, we say that two links $L_0, L_1 \subset S^3$ (with the same number of components) are \emph{strongly concordant} if there exist two disjoint proper embeddings of $S^1 \times [0,1]$ in $S^3 \times [0,1]$ interpolating between the two links.
We write $L_0 \sim L_1$ if $L_0$ and $L_1$ are strongly concordant.

$L^n$ will always denote a $n$-component link in $S^3$.
A link $L^{n+1}$ naturally gives a knot in $\#^n\es$ through the Ozsv\'ath--Szab\'o \emph{knotification} construction, which we will briefly recall here. Choose $2n$ distinct points $\{p_i,q_i\}_{i=1,\ldots, n}$ on $L^{n+1}$, such that if we identify each $p_i$ with $q_i$ we obtain a connected graph.

These points will be the attaching loci of $n$ $1$-handles; inside each of them, attach an oriented band connecting the two corresponding points.
Denote by $\kappa(L)$ the resulting knot in $\#^n\es$.
Using isotopies and handleslides, it is not hard to show (see~\cite[Prop. 2.1]{ozsvath2004holomorphicknot}) that the diffeomorphism class of the knot $\kappa(L)$ does not depend upon the specific choice of the points, hence knotification is well defined.
It is immediate to note that a knotified link will be null-homologous in $\#^n\es$:
indeed, by construction it intersects the co-core of each 1-handle exactly twice, and the intersections have opposite signs;
since the co-cores generate $H_2(\#^n\es;\Z)$, $\kappa(L)$ is null-homologous.

\begin{defi}
Given a knot $K \subset \#^n\es$, denote by $\mathcal{S}_n$ any set of $n$ embedded and pairwise disjoint 2-spheres $\{S_1, \ldots, S_n\}$, generating $H_2(\#^n\es;\Z) \cong \Z^n$. Then define the \emph{geometric winding number} of $K$ as 
\[
\geomi(K) = \min_{\mathcal{S}_n} \sum_{i=1}^n |K \cap S_i|.
\]
\end{defi}

Clearly, the knotification $\knoti{L} \subset \es$ of a $2$-component link $L$ will necessarily have $\geomi(\knoti{L}) \le 2$, and more generally $\geomi(\knoti{L^{n}}) \le 2(n-1)$.
Given $K \subset \es$, the connected sum with a local knot $K^\prime \subset S^3$ does not alter the geometric winding number. 
Hence any concordance bound obtained on $\geomi(K)$ is in fact an almost-concordance bound, using the terminology of \cite{celoria2018concordances}.

Recall that the Thuston norm is a seminorm $x$ on $H_2(Y,\partial Y; \R)$ for a compact $3$-manifold $Y$, introduced in~\cite{thurston1986norm}; the value of $x$ on a class $h \in H_2(Y,\partial Y; \Z)$ is given by $\displaystyle x(h) = \min \chi_-(S)$, where $S$ ranges over all properly embedded surfaces representing $h$, and $\chi_- (S) = \max\{-\chi(S),0\}$.

We will be mostly dealing with $2$-component non-split links $L\subset S^3$; in this case, by excision, we can view the Thurston norm as a map $x: H_2(S^3, L ;\R) \to \R$.
Moreover, since $H_2 (S^3, L;\Z)  \cong  H^1 (S^3\setminus L;\Z) \simeq H_1(S^3\setminus L;\Z)$, we will sometimes be sloppy and denote by $[h]$ both the classes in $H^1$ or $H_1$ (identified by the universal coefficients theorem using the basis of $H_1(S^3\setminus L)$ given by the meridians), and reserve the notation $\PD[h]$ for the Poincar\'e--Lefschetz dual of $[h]\in H^1(S^3\setminus L)$ in  $H_2(S^3,L)$.
With this convention, $H_2(S^3,L;\R)\cong \R^2$, with the basis given by the Poincar\'e--Lefschetz duals of the meridians of the components.

Another seminorm that can be considered on $H^1 (S^3 \setminus L)$ is given by link Floer homology, introduced by Ozsv\'ath and Szab\'o in~\cite{ozsvath2008holomorphic}, and recalled below. In its most basic form, link Floer homology is an homological invariant of links $L \subset S^3$, which categorifies the multivariable Alexander polynomial. 

If $L = K_0 \cup K_1$ and $\lk(K_0,K_1) = 0$, then the corresponding groups split according to a $\Z^2$-grading induced by elements of $\mathbb{H} = H_1(S^3\setminus L ; \Z)$:
\[
\HFL(L) = \bigoplus_{h \in \mathbb{H}} \HFL(L,h).
\]
To get a seminorm $y$ on $\mathbb{H}$, given $h^\prime \in H^1(S^3 \setminus L)$ define
\[
y(h^\prime) = \max_{\{h \in \mathbb{H} \mid \HFL(L,h) \neq 0\}} |\langle h^\prime,h\rangle|.
\]

These two seminorms are closely related; in fact~\cite[Theorem 1.1]{ozsvath2008link}, states that the Thurston polytope (and thus, the entire seminorm) of a link is determined by the link Floer polytope on $H^1(S^3 \setminus L; \R)$. More precisely:
\begin{equation}\label{e:relxey}
x(\PD[h]) + |\langle \mu_0,h\rangle| +|\langle \mu_1,h\rangle| = 2y(h),
\end{equation}
where $\mu_i$ is the meridian of $K_i$ for $i= 0,1$.\\

We can express the quantity $x(\PD[\mu_0])$ in more familiar terms.
\begin{lemma}\label{lemma:gtildeforever}
Let $L = K_0 \cup K_1\subset S^3$ be a link as above. Then
\begin{equation}\label{e:gtildeex}
\frac{1 + x(\PD[\mu_0])}{2} = y(\mu_0) = \min\{g(\Sigma)\;|\; \Sigma \hookrightarrow S^3 \setminus K_1 , \, \partial \Sigma = K_0\}.
\end{equation} 
\end{lemma}
\begin{proof}
The first equality follows immediately from Equation~\eqref{e:relxey}, so we are going to prove double inequalities between the first and last elements. For convenience, denote the term on the right in Equation~\eqref{e:gtildeex} by $\widetilde{g}(K_0,K_1)$.

Consider an embedded $\Sigma$ minimising $\widetilde{g}(K_0,K_1)$ (so cobounding $K_0$ in the exterior of $K_1$); clearly $[\Sigma] = \PD[\mu_0]$, hence $\widetilde{g}(K_0,K_1) = g(\Sigma) = \frac{1-\chi(\Sigma)}2 \ge \frac{1 + x(\PD[\mu_0])}{2}$. 

For the other direction, consider a surface $S$ realising $x(\PD[\mu_0])$; $S$ might be disconnected, and have multiple boundary components, which are simple and disjoint closed curves on $\partial S^3 \setminus \nu(L)$. Since the two components satisfy $\lk(K_0,K_1) = 0$, up to isotopy and attachments of annuli along boundary components, we can assume that these curves are a Seifert longitude on $\partial \nu(K_0)$ and a collection of meridians on $\partial \nu(K_1)$.

The signed count of these meridians needs to be $0$, in view of the condition $\lk(K_0,K_1)=0$.
By attaching other annuli connecting meridians with opposite signs (these annuli might be nested) we get a properly embedded surface $S'$.
By adding (possibly nested) annuli in near $\partial \nu(K_0)$ we can further assume that $\partial S'$ is connected, i.e. it is a Seifert longitude of $K_0$.

Note that adding annuli does not change the Thurston norm, so that $x(S') = x(S) = x(PD[\mu_0])$.
Moreover, $S'$ cannot have closed components that are not spheres or tori, since otherwise discarding them would decrease $x(S')$;
we discard all spheres and tori in $S'$, so that $S'$ is now connected.

The genus of $S'$ is precisely $\frac{1 + x(\PD[\mu_0])}{2}$.
\end{proof}

As an aside, recall that the \emph{concordance genus} of a knot $K \subset S^3$ is the minimal Seifert genus among all representatives in the concordance class of $K$. The left-hand side of Equation~\eqref{e:geomlowerboundthurstonconcordance} is an analogue of the concordance genus for $2$-component links.

\section{Heegaard Floer homology}\label{sec:hf}
Let $(Y,\ft)$ be a \spinc closed and orientable $3$-manifold, such that $c_1(\ft)$ is a torsion element in $H^2(Y;\Z)$; such a pair will be called a \emph{torsion} \spinc $3$-manifold.
We will only work with torsion \spinc 3-manifold in the paper, so $\Yt$ will always denote a torsion \spinc 3-manifold, unless explicitly stated otherwise.

To $\Yt$, Oszv\'ath and Szab\'o associate two $\Q$- and $\Z/2\Z$-graded $\F[U]$-modules, $\HF\Yt$ and $\tHF\Yt$~\cite{ozsvath2004holomorphic, ozsvath2004holomorphicpropr}.
The latter denotes (the plus flavour of) Heegaard Floer homology of $\Yt$ with fully twisted coefficients, while the former denotes (the plus flavour of) Heegaard Floer homology of $\Yt$.
When $Y$ is a rational homology sphere, there is no twisting, and $\HF\Yt = \tHF\Yt$.

We write $\HFe\Yt$ and $\HFo\Yt$ for the parts of $\Z/2\Z$-degree 0 and 1 respectively.
We also write $\HF(Y)$ as a shorthand for $\bigoplus_{\ft\in\Spinc(Y)} \HF\Yt$, and $\tHF(Y)$ as a shorthand for $\bigoplus_{\ft\in\Spinc(Y)} \tHF\Yt$ (here we are summing over all \spinc structures, not just the torsion ones).
Moreover, there is a well-defined quotient $\HFred\Yt$ of $\HF\Yt$, called the \emph{reduced} part of $\HF\Yt$;
if a rational homology sphere $Y$ has $\HFred\Yt = 0$ for each \spinc structure $\ft$, we say that it is an L-space.

A \spinc cobordism $(W,\fs)$ from $\Yt$ to $(Y',\ft')$ induces a map $F_{W,\fs}:\HF\Yt\to\HF(Y',\ft')$;
there is an analogue for the fully twisted version as well.
To the underlying smooth cobordism $W$ we can associate $F_W: \HF(Y) \to \HF(Y')$, by summing over all \spinc structures on $W$.

Let $M$ be a compact 3-dimensional manifold with torus boundary; consider three slopes $s_0, s_1, s_\infty$ on $\partial M$ such that $s_0 \cdot s_1 = s_1 \cdot s_\infty = s_\infty\cdot s_0 = -1$;
the three 3-manifolds $Y_0, Y_1, Y_\infty$ obtained by filling $M$ with slopes $s_0, s_1, s_\infty$ respectively, are said to form a \emph{triad}.
The key example of triad is when $M$ is the complement of a null-homologous knot $K$ in a 3-manifold, $s_\infty$ is the slope of the meridian of $K$, and $s_0$ and $s_1 = s_0 + s_\infty$ are consecutive integral slopes.
A triad gives rise to three cobordisms $W_\infty, W_0, W_1$, where $W_\infty: Y_0 \leadsto Y_1$ is obtained from $Y_0$ by attaching a single 2-handle ($W_0$ and $W_1$ are defined analogously, cyclically permuting the indices).
The associated maps fit into an exact triangle:
\[
\xymatrix{
\HF(Y_1) \ar[rr]^{F_{W_0}} & & \HF(Y_\infty)\ar[dl]^{F_{W_1}}\\
 & \HF(Y_0)\ar[ul]^{F_{W_\infty}}
}
\]
If $b_1(Y_0) = 1$ and $b_1(Y_1) = b_1(Y_\infty) = 0$, there is also a twisted coefficient version of the triangle above:
\[
\xymatrix{
\HF(Y_1)[t,t\inv] \ar[rr]^{\underline{F}_{W_0}} & & \HF(Y_\infty)[t,t\inv]\ar[dl]^{\underline{F}_{W_1}}\\
 & \tHF(Y_0)\ar[ul]^{\underline{F}_{W_\infty}}
}
\]
in which the maps $\underline{F}_{W_0}, \underline{F}_{W_1}, \underline{F}_{W_\infty}$ are a suitably adapted version of the maps $F_{W_0}, F_{W_1}, F_{W_\infty}$, and $\HF(Y)[t,t\inv]$ is a shorthand for $\HF(Y)\otimes_\F\F[t,t\inv]$.

When $\Yt$ is a rational homology sphere, from $\HF\Yt$ Ozsv\'ath and Szab\'o extract a numerical invariant $d\Yt$, the \emph{correction term} of $\Yt$~\cite{OSz-absolutely};
this was also extended to 3-manifolds `with standard $\HFoo$', to define the \emph{bottom-most correction term} $d_b\Yt$; e.g.~if $b_1(Y) \le 2$, $Y$ automatically has standard $\HFoo$~\cite[Theorem 10.1]{ozsvath2004holomorphicpropr}.

Stefan Behrens and the second author generalised this construction using twisted coefficients~\cite{behrens2015heegaard}; from $\tHF\Yt$ one can then define the \emph{twisted correction term} $\dtw\Yt$.
This is a rational number associated to $\Yt$;
it is invariant under \spinc rational homology cobordism, and additive under connected sums; that is, $\dtw(Y\#Y', \ft\#\ft') = \dtw(Y,\ft)+\dtw(Y',\ft')$.
Moreover, it agrees with the usual untwisted version for rational homology spheres.

We will be using the following additional property of twisted correction terms.
\begin{thm}[{\cite[Proposition 4.1]{behrens2015heegaard}}]\label{t:dtwcobordism}
Let $\Yt$, $(Y',\ft')$ be torsion \spinc $3$-manifolds, and $(W,\fs)$
a negative semi-definite \spinc cobordism from $\Yt$ to $(Y',\ft')$.
Assume moreover that the inclusion $Y\hookrightarrow W$ induces an injection $H_1(Y;\Q)\to H_1(W;\Q)$.
Then
\[
c_1(\fs)^2 + b_2^-(W) + 4\dtw\Yt + 2b_1(Y) \le 4\dtw(Y',\ft') + 2b_1(Y').
\]
\end{thm}

We will also use the following computations;
we will omit the \spinc structure from the notation when there is a unique torsion \spinc structure, i.e.~when $H_1$ of the 3-manifold is torsion-free.

\begin{prop}[{\cite[Theorem 6.1]{behrens2015heegaard}}]\label{p:dtwcomputations}
If $\Sigma$ is a closed orientable surface of genus $g$, then $\dtw(\Sigma\times S^1) = (-1)^{g+1}/2$; in particular,
\[
4\dtw(\Sigma\times S^1) + 2b_1(\Sigma\times S^1) = 8\left\lceil \frac g2\right\rceil.
\]
\end{prop}

We now give a way to index torsion \spinc structures on certain 3-manifolds, following~\cite[Section 2.4]{OSz-integralsurgeries}; we will abide by this labelling convention for the rest of the paper.

Suppose $Z$ is a closed 3-manifold with torsion-free $H_1(Z)$ (e.g. $Z = S^3$), and that $K\subset Z$ is a null-homologous knot.
Consider the 4-manifold $X_n(K)$ obtained by attaching a $2$-handle to $Z\times [0,1]$ along $K\times\{1\}$, with framing $n$;
for convenience, we let $Z_0 = Z\times\{0\}$.
An orientation of $K$ determines a generator $A$ of $H_2(X_n(K), Z_0)$;
the \spinc structure $\ft_i$ on $S^3_n(K)$ is the restriction of the unique \spinc structure $\fs_i$ on $X_n(K)$ such that $\langle c_1(\fs_i), A\rangle = n - 2i$ and $\fs_i|_{Z_0}$ is torsion.
(While, a priori, this construction depends on the choice of an orientation on $K$, this choice turns out to be immaterial;
this is because of conjugation symmetry in Heegaard Floer homology.)
Note that, when $Z$ is an integer homology sphere, the last condition is automatically satisfied.

Finally, we recall a way to compute correction terms of positive surgeries along knots (and especially along connected sums of torus knots).

\begin{thm}[\cite{rasmussen2003floer, NiWu}]\label{t:NiWu}
Fix a knot $K$ and a positive integer $n$. There is a non-increasing sequence of non-negative integers $\{V_i(K)\}_{i\ge 0}$ such that:
\[
d(S^3_n(K), \ft_i) = -2\max\{V_i(K), V_{n-i}(K)\} + \frac{(n-2i)^2}{4n} - \frac14.
\]
\end{thm}

In the notation of the latest theorem, we have the following.

\begin{prop}[{\cite[Example 3.9]{behrens2015heegaard}}]\label{p:0surgery}
Let $K$ be a knot in $S^3$; then $\dtw(S^3_0(K)) = d_b(S^3_0(K)) = -\frac12+2V_0(-K)$.
\end{prop}

For positive torus knots, the sequence $\{V_i(T_{p,q})\}$ can be computed in terms of the arithmetics of $p$ and $q$ as follows (see~\cite[Equation (5.1)]{borodziklivingston}):
let $\Gamma_{p,q}$ be the semigroup generated by $p$ and $q$, i.e. $\Gamma_{p,q} = \{hp+kq \mid h,k \in \Z_{\ge 0}\}$;
then
\begin{equation}\label{e:semigroup}
V_i(T_{p,q}) = |\Gamma_{p,q} \cap [0,g(T_{p,q})-i)|, 
\end{equation}
where $g(T_{p,q}) = \frac{(p-1)(q-1)}2$ is the genus of $T_{p,q}$, and $|\cdot|$ denotes the cardinality of a set.

More generally, a similar computation works for connected sums of torus knots, of which we will only be using a special case; specifically, we claim that for every choice of positive integers $a$, $b$, and $n > 1$, $V_i(T_{n,an+1}\#T_{n,bn+1}) = V_i(T_{n,(a+b)n+1})$.
For completeness, we sketch the proof.

To any algebraic knot one can associate the \emph{multiplicity sequence}: in brief, this is a non-increasing sequence of positive integers that keeps track of how the knot is resolved by blowups.
The multiplicity sequence of $T_{n,kn+1}$ is of length $k$, and its entries are all $n$, i.e.~the sequence is $[n,\dots,n]$;
then, the concatenation of the multiplicity sequences of $T_{n,an+1}$ and $T_{n,bn+1}$ is the multiplicity sequence of $T_{n,(a+b)n+1}$;
in the notation of Bodn\'ar--Nem\'ethi~\cite[Theorem 5.1.3]{bodnar2016lattice}, this says $H_{T_{n,an+1}}\diamond H_{T_{n,bn+1}} = H_{T_{n,(a+b)n+1}}$;
since the function $H_K$ determines the sequence $\{V_i(K)\}_{i\ge 0}$, the claim is proved.

We state this explicitly when $a=b=2$.

\begin{lemma}\label{l:dT3737}
For each $m>0$ and $0 \le i < m$,
\[
d(S^3_m(T_{n,2n+1}\#T_{n,2n+1}),\ft_i) = d(S^3_m(T_{n,4n+1}),\ft_i).
\]
\end{lemma}

We conclude the section with another lemma that will be useful later on.

\begin{lemma}\label{l:V0}
The following equalities hold:
\begin{enumerate}\itemsep -1pt
\item[(I)$\,\,$] $V_0(T_{2n,2n+1}) = \frac12n(n+1)$;
\item[(II)$\,$] $V_0(T_{2n,8n+1}) = 2n^2$;
\item[(III)] $V_0(T_{2n+1,8n+5}) = 2n(n+1)$.
\end{enumerate}
\end{lemma}

\begin{proof}
We prove the first equality, and only sketch the proofs of the other two.
Since the genus of $T_{2n,2n+1}$ is $n(2n-1)$, from Equation \eqref{e:semigroup} above we know that we must count how many elements in the semigroup generated by $2n$ and $2n+1$ are strictly smaller than $n(2n-1)$.
For elements in $\Gamma_{p,q}$ that are less than $pq$, the representation $hp+kq$ is unique, therefore we only need to count pairs $(h,k)$ such that $2hn+k(2n+1) < n(2n-1)$.
One easily shows that this number is
\[
|\Gamma_{2n,2n+1} \cap [0,2n^2-n)| = \sum_{k=0}^{n-1}\left\lceil \frac{2n^2-n - k(2n+1)}{2n} \right\rceil = \sum_{k=0}^{n-1} (n-k) = \frac12 n(n+1).
\]

For points (II) and (III), the computation is very similar; in the first case, one has
\begin{align*}
|\Gamma_{2n,8n+1} \cap [0,8n^2-4n)| &= \sum_{k=0}^{n-1}\left\lceil \frac{8n^2-4n - k(8n+1)}{2n} \right\rceil \\
&= \sum_{k=0}^{n-1} 2(2n-1-2k) = 2n(n+1) - 2n = 2n^2.
\end{align*}

In the second case, one has
\begin{align*}
|\Gamma_{2n+1,8n+5} \cap [0,8n^2+4n)| &= \sum_{k=0}^{n-1}\left\lceil \frac{8n^2+4n - k(8n+5)}{2n+1} \right\rceil \\
&= \sum_{k=0}^{n-1} 4(n-k) = 2n(n+1).\qedhere
\end{align*}
\end{proof}

\section{Bounds on the Thurston norm}\label{sec:dimostrazione}

The goal of this section is to prove the following generalisation of Theorem~\ref{t:main}.
The setup is the following: $L = K_0 \cup K_1$ is a link with $\lk(K_0,K_1) = 0$ in $S^3$, and $Y_{L,n}$ is the 3-manifold obtained by doing $n$-surgery along $K_1$, and $0$-surgery on $K_0$; as before, recall that $\mu_0$ is the meridian of $K_0$. 

\begin{thm}\label{t:geomlowerboundthurston}
If $Y_{L,n}$ is obtained by doing $(0,n)$-surgery along $L$ as above, then:
\begin{equation}\label{e:geomlowerboundNthuston}
\min_{L^\prime \sim L } \left\lceil \frac{x(\PD[\mu_0^\prime]) +1}4\right\rceil  \ge 
\frac12\max_{\ft\in\Spinc(Y_{L,n})}\{\dtw(Y_{L,n},\ft) + \dtw(-Y_{L,n},\ft) + 1\}.
\end{equation}
\end{thm}

We state here the specialised version in which $K_0$ is the unknot; if we perform $0$-surgery on $K_0$, $K_1$ becomes a null-homologous knot in $\es$, that we denote with $K$.
In this case, $Y_{L,n}$ is also obtained as $n$-surgery along $K$;
since we want to emphasise that $Y_{L,n}$ should be regarded as a surgery along $K\subset\es$, we write $Y_{K,n}$ instead of $Y_{L,n}$.

\begin{thm}\label{t:geomlowerbound}
If $Y_{K,n}$ is obtained by doing $n$-surgery along $K$, then:
\begin{equation}\label{e:geomlowerboundN}
\left\lceil\frac{\geomi(K)}4\right\rceil \ge \frac12\max_{\ft\in\Spinc(Y_{K,n})}\{\dtw(Y_{K,n},\ft) + \dtw(-Y_{K,n},\ft) + 1\}.
\end{equation}
\end{thm}

We also observe that Theorems~\ref{t:mainthurston} and~\ref{t:main}  are an immediate corollary of Theorem~\ref{t:geomlowerboundthurston}, obtained by setting $n=1$ and using Lemma~\ref{lemma:gtildeforever}.\\

Finally, we show that, in some special cases, we can compute the right-hand side of~\eqref{e:geomlowerbound} more explicitly.

\begin{cor}\label{c:surgery}
Suppose $Y_{K,1}$ is obtained as $0$-surgery along a knot $J\subset S^3$. Then
\begin{equation}\label{e:geomlowerbound1}
\left\lceil\frac{\geomi(K)}4\right\rceil \ge V_0(J) + V_0(-J).
\end{equation}
\end{cor}

\begin{proof}
If $Y_{K,1} = S^3_0(J)$, then, by Proposition~\ref{p:0surgery}:
\[
\dtw(Y_{K,1},\ft_0) = \dtw(S^3_0(J)) = d_b(S^3_0(J)) = d_{-1/2}(S^3_0(J)) = -\frac12 + 2V_0(-J).
\]
Since $-Y_{K,1} = S^3_0(-J)$, one also has
\[
\dtw(-Y_{K,1},\ft_0) = -\frac12 + 2V_0(J),
\]
and substituting them in~\eqref{e:geomlowerbound} yields the desired inequality.
\end{proof}

We now turn to the proof of Theorem~\ref{t:geomlowerboundthurston}.
To this end, we set up some notation and give some preliminary constructions.

Suppose that $\Sigma'$ is the closed surface obtained by capping off a minimal genus surface cobounding $K_0$ in the complement of $K_1$ with the core of the $0$-framed handle.

\begin{figure}
\labellist
\pinlabel $0$ at 280 275
\endlabellist
\includegraphics[width = 0.7\textwidth]{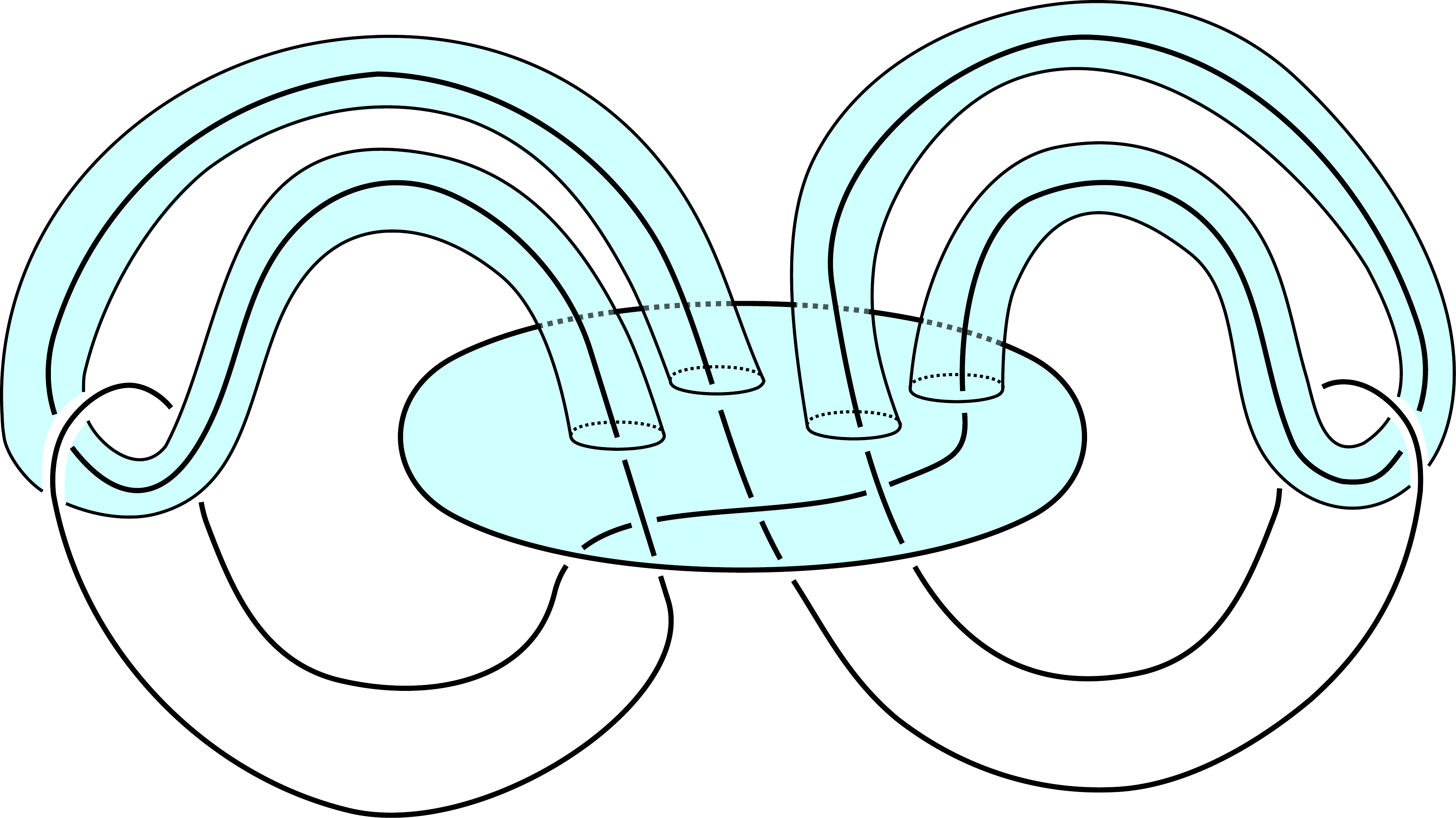}
\caption{The closed surface $\Sigma_1$, in the case of a link $K \cup \bigcirc$ with $0$ linking number. $K$ intersects the sphere obtained by capping off the disk bounded by $\bigcirc$ four times.
Note that some nesting of the tubes might be necessary.}
\label{fig:tubing}
\end{figure}

To shorten up the notation, we also let $Y := Y_{L,n} = S^3_{(0,n)} (K_0\cup K_1)$.

Since $\Sigma' \subset S^3_0(K_0)$ is disjoint from $K_1$, $\Sigma'$ survives in $Y$, and its homology class generates $H_2(Y) \cong \Z$.
With a slight abuse of notation, we still denote it with $\Sigma'\subset Y$.

Consider now the trivial cobordism $Y\times I$; then $\Sigma'\times\{1/2\}$ is a surface in $Y\times I$ with trivial normal bundle (e.g.~because it is trivialised by $\partial/\partial t$, where $t$ parametrises the interval $I$).
We denote $\Sigma'\times\{1/2\}$ by $\Sigma$.
Call $W$ the 4-manifold $Y\times I \setminus N$, where $N$ is a regular neighbourhood of $\Sigma$;
since $\Sigma$ has trivial normal bundle, $N$ is diffeomorphic to $\Sigma\times D^2$, and $-\partial N$ is diffeomorphic to $\Sigma\times S^1$.

We view $W$ as a cobordism from $Y\sqcup -Y$ to $\Sigma\times S^1$.
We want to apply Theorem~\ref{t:dtwcobordism};
in order to do so, we need the following lemma.

\begin{lemma}\label{l:injectiveH1}
In the notation above, the inclusion $Y\sqcup -Y \hookrightarrow W$ induces an injective map $H_1(Y\sqcup -Y;\Q) \to H_1(W;\Q)$.
\end{lemma}
\begin{proof}
Identify $H_1(Y \sqcup -Y;\Q)$ with $H_1(Y;\Q) \oplus H_1(Y;\Q)$ in the natural way, using the inclusions $Y\times\{0\}, Y\times\{1\} \hookrightarrow Y\times I$.
Suppose $H_1(Y\sqcup -Y;\Q) \ni (a,b) \mapsto 0 \in H_1(W;\Q)$ under map induced by the inclusion.
Taking a multiple if necessary, we can assume that $(a,b)$ is in fact an integral class.

Since $W\subset Y\times I$, in particular $(a,b)$ vanishes in $H_1(Y\times I;\Q) = H_1(Y;\Q)$;
however, with the choice we made, the map $H_1(Y\sqcup -Y;\Q) \to H_1(Y;\Q)$ can be identified with the map $(a,b) \mapsto a+b$, and therefore we obtain that $b = -a$.

Since $a$ is integral, we can represent $a$ by a simple closed curve $\alpha$ which meets $\Sigma\subset Y$ transversely in a collection of signed points $P$.
Call $p$ the signed count of points in $p$, and note that $p\neq 0$, since $a$ is a non-zero class in $H_1(Y\times I; \Q)$, which therefore pairs nontrivially with $[\Sigma]$;
in fact, changing the sign of $a$, we can assume that $p > 0$.
The surface $\alpha \times I$ bounds $(a,-a)$, and meets $\Sigma\subset Y\times I$ transversely in $P\times\{1/2\}$.
It follows that ${\alpha \times I} \cap W$ gives the relation $H_1(\de W;\Q) \ni (a,-a,p\cdot m) \mapsto 0 \in H_1(W;\Q)$, where $m$ is the meridian of $\Sigma$, i.e.~(up to orientation) the curve $\{*\} \times S^1 \subset \Sigma\times S^1 = \de N$.
In particular, there is a surface $(F_0,\partial F_0)$, properly embedded in $(W,\partial W)$, whose boundary is $(\alpha,-\alpha,m_1\cup\dots\cup m_p)\subset Y\sqcup Y \sqcup \de N$.

Since we assumed that $(a,-a)$ vanishes in $H_1(W)$, there exists a surface $(F_1,\partial F_1)$ properly embedded in $(W,Y\sqcup -Y)$, whose boundary is $(\alpha,-\alpha) \subset Y\sqcup -Y$.
Gluing $F_0$ and $F_1$ along their common boundary, we obtain a surface $(F,\de F)$, properly embedded in $(W, \de N)$.
Capping off $F$ with $p$ disc fibres $\{q\}\times D^2 \subset N$, we obtain that $\Sigma$ has a (rationally) dual surface $\widehat F$ in $Y\times I$, i.e.~$\widehat F$ and $\Sigma$ meet transversely $p$ times.

But then $[\Sigma]$ and $[\widehat F]$ are two homology classes in $H_2(Y\times I)$ intersecting non-trivially, which is clearly a contradiction, since the intersection form on $H_2(Y\times I)$ is trivial.
\end{proof}

\begin{proof}[Proof of Theorem~\ref{t:geomlowerboundthurston}]
We start by observing that the right-hand side of~\eqref{e:geomlowerboundNthuston} is an invariant of the strong concordance class of $L$;
in fact, if $L^\prime$ is concordant to $L$, then $Y_{L^\prime,n}$ is (integrally) homology cobordant to $Y_{L,n}$, and therefore $\dtw(\pm Y_{L^\prime,n},\ft_i) = \dtw(\pm Y_{L,n},\ft_i)$.

We view $W$ as a cobordism from $Y\sqcup -Y$ to $\Sigma\times S^1$.
Let $\gamma \subset W$ be an embedded arc connecting the two boundary components $Y$ and $-Y$, and remove a small regular neighbourhood of $\gamma$ from $W$, to obtain $W'$.
This is now a cobordism from $Y\#-Y$ to $\Sigma\times S^1$, and Lemma~\ref{l:injectiveH1} above implies that the inclusion of $Y\# -Y$ in $W'$ induces an injective map at the level of $H_1$ with rational coefficients.

Now call $\fs_i$ the restriction to $W'$ of the unique \spinc structure on $Y\times I$ that restricts to $\ft_i$ on $Y\times \{0\}$.
Note that $\fs_i$ is uniquely defined, since $Y\times I$ is a product, and that $\fs_i$ also restricts to $\ft_i$ on $Y\times\{1\}$.
Observe also that $\fs_i$ restricts to the unique torsion \spinc structure on $\partial N$, and that $c_1^2(\fs_i) = 0$, since the intersection form of $W'$ is trivial.

Thanks to Lemma~\ref{l:injectiveH1}, and since $W$ is negative semidefinite, we are in the assumptions of Theorem~\ref{t:dtwcobordism} to the \spinc structures $\fs_i$; this yields:
\begin{align*}
4\dtw(Y,\ft_i)+4\dtw(-Y,\ft_i)+4 &= c_1^2(\fs_i) + b^-(W') + 4\dtw(Y \# -Y, \ft_i\#\ft_i) + 2b_1(Y \# -Y)  \le\\
& \le  4\dtw(\Sigma\times S^1) + 2b_1(\Sigma\times S^1) = 8\left\lceil\frac g2\right\rceil;
\end{align*}
where the last equality is Proposition~\ref{p:dtwcomputations}.
\end{proof}

The proof of Theorem~\ref{t:geomlowerbound} is a special case of the previous one, in which the component $K_0$ is assumed to be unknotted, hence $S^3_0(K_0) = \es$. 
We can assume that $K_1$ intersects the sphere $\Sigma'' = S^2\times\{1\}$ transversely in $2g = \geomi(K)$ points;
by tubing $\Sigma''$ along $K$, we obtain a surface $\Sigma' \subset \es$ disjoint from $K$, as in Figure~\ref{fig:tubing}.
We note here that $\Sigma'$ has genus $g$, and that it represents the generator of $H_2(\es)$
\begin{proof}[Proof of Theorem~\ref{t:geomlowerbound}]
The proof is readily obtained by applying Theorem~\ref{t:geomlowerboundthurston}, and reinterpreting the result in the light of Lemma~\ref{lemma:gtildeforever}.
\end{proof}

\subsection{The case of knots in $\#^m\es$}

Theorem~\ref{t:geomlowerbound} extends to the case of null-homologous knots in $\#^m\es$ as follows.
Given a null-homologous knot $K$ in $\#^m\es$, let $\{S_1,\dots,S_m\}$ be a collection of $m$ pairwise disjoint spheres in $\#^m\es$ whose homology classes generate $H_2(\#^m\es)$;
suppose $K$ intersects $S_i$ transversely for each $i$; denote with $\geomi_i(K)$ the geometric intersection of $K$ and $S_i$, and $\geomi(K) = \sum_{i=1}^m \geomi_i(K)$.

\begin{thm}
Let $K$ be a null-homologous knot in $\#^m\es$, 
Then
\[
\sum_{i=1}^n \left\lceil \frac{\geomi_i(K)}4 \right\rceil \ge \frac12 \max_{\ft\in\Spinc(Y_{K,n})} \{\dtw(Y_{K,n}, \ft) + \dtw(-Y_{K,n},\ft) + 1\}.
\]
\end{thm}

Note that this can be used to give a (quite coarse) concordance lower bound on $\geomi(K)$.
Indeed, if $\{S_1,\dots,S_m\}$ is a collection of spheres that minimises the total geometric intersection, then
\[
\geomi(K) = \sum_{i=1}^m \geomi_i(K) \ge \sum_{i=1}^m 4\left(\left\lceil \frac{\geomi_i(K)}4 \right\rceil - \frac12\right) = 4\sum_{i=1}^m \left\lceil \frac{\geomi_i(K)}4 \right\rceil - 2m,
\]
so that
\[
\geomi(K) \ge 2\max_{i=0,\dots,n-1} \{\dtw(Y_{K,n}, \ft_i) + \dtw(-Y_{K,n},\ft_i) + 1\} - 2m.
\]
As above, the right-hand side of the latter inequality is invariant under concordance, so we get a concordance lower bound.
The proof is very similar to the proof of Theorem~\ref{t:geomlowerbound}, therefore we only outline the differences here.

\begin{proof}[Proof (sketch)]
From the collection $\{S_1,\dots,S_m\}$ we construct $m$ pairwise disjoint, orientable surfaces $\Sigma_1,\dots,\Sigma_m$ in $Y_{K,n}\times \{1/2\} \subset Y_{K,n}\times I$ by tubing $S_1,\dots S_m$ along $K$, as in the proof of Theorem~\ref{t:geomlowerbound}.
The genus of $\Sigma_i$, which is obtained from $S_i$ by tubing along $K$, is exactly $g_i := \frac{\geomi_i(K)}2$.

We construct a cobordism $W$ from $Y_{K,n}\sqcup -Y_{K,n}$ to $\sqcup_i \Sigma_i\times S^1$ by removing the tubular neighbourhood of $\Sigma_1, \dots, \Sigma_m$ from $Y_{K,n}\times I$.

We now claim that Lemma~\ref{l:injectiveH1} still holds for the cobordism $W$ we just constructed.
Again, we can suppose that we have a class $a \neq 0\in H_1(Y_{K,n})$ such that $(a,-a)\in H_1(Y_{K,n})\oplus H_1(Y_{K,n})$ vanishes under the map induced by the inclusion $Y_{K,n}\sqcup -Y_{K,n} \hookrightarrow W$, and that $a$ is represented by a curve $\alpha$.
The only difference in the two proofs is the following: in the proof of Lemma~\ref{l:injectiveH1} we had only one surface $\Sigma$, and we argued that the algebraic intersection number between $\Sigma$ and $\alpha\times I \subset Y\times I$ was non-zero by assumption that $[\alpha] \neq 0 \in H_1(Y)$;
in the new setup, we know that \emph{for some index} $i$, the intersection between $\Sigma_i$ and $\alpha\times I \subset Y_{K,n}\times I$ is non-zero;
now work in $Y_{K,n}\times I \setminus N(\Sigma_i)$, and run the same argument.

The rest of the proof applies verbatim.
\end{proof}

\section{The essential case}\label{sec:essential}

In this section, we will see how to deal with the case of $K$ essential in $\es$, and more specifically when $[K] = w\cdot[\{*\}\times S^1] \in H_1(\es)$, for some \emph{even} integer $w$; without loss of generality, we assume that $w$ is positive.
To get a knot in a class divisible by $2$, one can simply take a satellite of $K$ using a pattern with even winding number, e.g.~a 2-cable.
To this end, we will combine the topological construction from the previous section with arguments from~\cite{levine2015nonorientable}.
As in the null-homologous case, this will turn out to be a \emph{concordance} bound for $K$, i.e. a lower bound for $\geomi_\C(K)$.

We note that the setup is slightly different in this case; for instance, we do not have a well-defined way to associate an integer to a framing.
To remedy this, we fix a handlebody presentation of $(\es,K)$, where $\es$ is viewed as the boundary of $D^3\times S^1$, and the latter is obtained by carving a disc from $B^4$; as usual, the carved disk will be denoted by a dotted circle.
Such a presentation for $(\es,K)$ gives a bijection between framings of $K$ and the integers;
we will be sloppy and use this bijection without explicitly mentioning the presentation.

Let $Y_n(K)$ be the 3-manifold obtained by doing $n$-surgery along $K$.
The following proposition is well-known, and we shall omit the proof.

\begin{prop}\label{p:H1Y}
The $3$-manifold $Y_n(K)$ is a rational homology sphere; its first homology group $H_1(Y_n(K))$ is generated by the classes of the meridians of the attaching curve of the dotted circle and of $K$; $H_1(Y_n(K)) \cong \Zmod{d} \oplus \Zmod{d'}$, where $d = \gcd(n,w)$, and $dd' = w^2$.
\end{prop}

We note here that, in fact, $\gcd(n,w)$ is independent of the chosen presentation of $(\es,K)$.
From now on, we restrict to the case when $\gcd(n,w) = 1$, and hence $H_1(Y_n(K))$ is cyclic of order $w^2$.
Moreover, since the exact value of $n$ will not play any significant role, we drop it from the notation, and we write $Y$ in place of $Y_n(K)$.
In fact, under the assumption above, $H_1(Y)$ is generated by the class $[\mu]$ of the meridian of $K$; the meridian of the 2-handle is homologous to $w[\mu]$. (The latter class is always the generator of the metaboliser of $H_1(Y)$ associated to the obvious rational homology ball filling of $Y$.)

Since we assumed that $w$ is even, $H_1(Y)$ is cyclic of even order, and thus every element has an \emph{opposite}: the opposite of the element $k[\mu]\in H_1(Y)$ is the element $(k+w^2/2)[\mu] \in H_1(Y)$.
This gives an involution of $H_1(Y)$ without fixed points, and, correspondingly, the set of \spinc structures on $Y$ comes equipped with a fixed-point--free involution, that associates to $\ft\in\Spinc(Y)$ the \spinc structure $\fto = \ft + w^2/2\cdot \PD([\mu])$.

In the notation of~\cite{levine2015nonorientable}, $w^2/2\cdot \PD([\mu])$ is called $\varphi$, i.e. a 2-torsion class in $H^2(Y)$; in our setting, this characterises $\varphi$ uniquely.
Incidentally, we note here that $c_1(\fto) = c_1(\ft)$. We remark here that being opposite is not to be confused with being conjugate; both conjugation and opposition are involutions on the set of \spinc structures of $Y$, but the former has fixed points (the two spin structures on $Y$), preserves the value of the correction term, and changes the sign of the first Chern class.

\begin{thm}\label{t:geomi-essential}
With the notation set up as above, we have:
\[
\geomi(K) \ge 2\max_{\ft\in\Spinc(Y)} \{d(Y,\ft) - d(Y,\fto)\}.
\]
\end{thm}

\begin{proof}
Consider a sphere $\Sigma'' = S^2\times\{1\}$, and suppose that it meets $K$ transversely $h := \geomi(K)$ times.
Note that $h$ is even, since $\geomi(K) \equiv w \equiv 0 \pmod 2$ by assumption.

By tubing along $K$ we can construct a surface $\Sigma'\subset \es\setminus K$ from $\Sigma''$ by tubing along $K$ as in the proof of Theorem~\ref{t:geomlowerbound}.
(Note that we are using in a crucial way that the class of $K$ is even.)
Here, however, $\Sigma'$ will be non-orientable, and $b_1(\Sigma_1; \F_2) = h$;
$h$ is referred to as the non-orientable genus of $\Sigma'$.

Since $\Sigma'$ lives in the complement of $K$, we can view $\Sigma'$ as lying in any surgery along $K$, and in particular in $Y$.
As we did above, we push it in $Y\times I$ at level $1/2$, obtaining $\Sigma \subset Y\times I$.

We can now apply~\cite[Theorem A]{levine2015nonorientable}, which asserts that $h \ge 2\Delta$, where $2\Delta$ is exactly the right-hand side of the inequality we want to prove.
Since $h = \geomi(K)/2$, we are done.
\end{proof}

We note here that, in the notation of the proof above, \cite[Theorem A]{levine2015nonorientable} also asserts that $2h \ge 4\Delta + |e(\Sigma)|$, where $e(\Sigma)$ is the Euler number of $\Sigma$;
however, since $\Sigma$ lives in $Y\times\{1/2\}$, it is displaceable, and in particular $e(\Sigma) = 0$.
In particular, the seemingly stronger inequality does not give a better lower bound.

\section{Examples}\label{sec:esempi}

This section is devoted to some sample computations of the obstruction from Equation~\eqref{e:geomlowerbound}.
After warming up with a baby-case, we obtain an example where the lower bound~\ref{t:geomlowerbound} is sharp and non-trivial, while Schneiderman's obstruction (whose definition is recalled below) vanishes;
then, we construct an infinite family of knots such that the lower bound of Theorem~\ref{t:geomlowerbound} is sharp and unbounded.

We start by considering the knot $W$ in $\es$ obtained by doing $0$-surgery along one of the components of the Whitehead link. Equivalently $W$ can be thought of as the knotification of the Hopf link.

\begin{figure}[h]
\labellist
\pinlabel $0$ at -30 275
\endlabellist
\includegraphics[width = 0.25\textwidth]{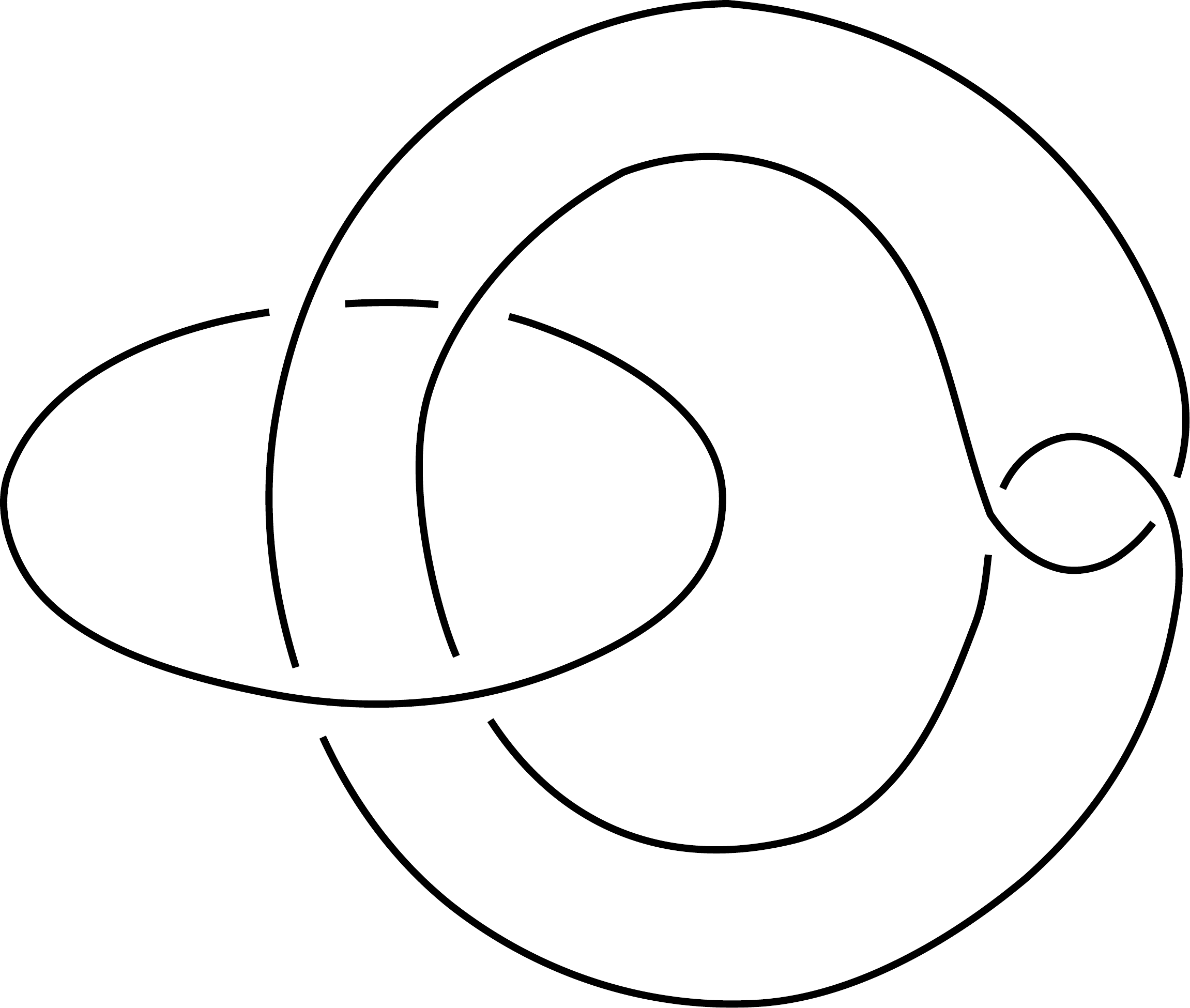}
\caption{The knot $W \subset \es$.}
\label{f:whitehead-link}
\end{figure}

We say that a knot is local if it is contained in a 3-ball in $\es$.

\begin{prop}
The knot $W$ is not concordant to a local knot.
\end{prop}
\begin{proof}
Obviously, a knot $K$ is a local knot if and only if $\geomi(K) = 0$.
As noted above, our lower bound on $\geomi$ is a concordance invariant, therefore it suffices to prove that the lower bound for $W$ does not vanish.
To this end, we observe that $+1$-surgery along $W$ yields the 3-manifold obtained as $0$-surgery along the trefoil knot $T$.
Since $V_0(T) = 1$, $V_0(-T) = 0$, by Corollary~\ref{c:surgery} the lower bound on $\geomi(W)\ge 2$, and therefore $W$ is not concordant to a local knot.
\end{proof}

\subsection{Comparing with Schneiderman's bound}

We recall the construction of Schneiderman's invariant $\mu$ from~\cite{schneiderman2003algebraic}.
While his setup is more general, we restrict to the case of knots in $\es$.
Here, the invariant of a null-homologous knot $K\in \es$ takes the form of a polynomial $\mu(K) \in t\cdot\Z[t]$.
It is an invariant of (locally flat) topological concordance, and it can be computed in the following way.

Since $K$ is null-homologous, there is a regular homotopy of $K$ to the unknot.
This gives rise to an immersed disc $j: D \looparrowright \es\times I$;
generically, such a disc will have only double points.
To each double point $p$ correspond a sign $\sigma(p)$ and a generator $\gamma(p)$ of $\pi_1(j(D))$, determined up to inverse.
The generator $\gamma(p)$, in turn, gives a homotopy class in $\es\times I$, and hence an element $w(p)$ in $\pi_1(\es\times I) = \Z$, well-defined up to sign.
The invariant $\mu(K)$ is computed as
\[
 \mu(K) = \sum_{p \,|\,w(p)\neq 0} \sigma(p)\cdot t^{|w(p)|}. 
\]
The following proposition was suggested to us by Mark Powell.
\begin{prop}
The degree of $\mu(K)$ gives a lower bound for $\geomi(K)$. More precisely, 
\[
\geomi(K) \ge 2\deg\mu(K). 
\]
\end{prop}

\begin{proof}
Choose a representation of $K$ as one component of a 2-component link, one of whose component is a dotted unknot; since $K$ is null-homologous, there is a sequence of crossing changes in this projection, involving only crossings of $K$ with itself, that changes the link to the unlink.
Associated to this sequence of crossing changes, comes a regular homotopy from $K$ to the unknot in $\es$, and a corresponding immersed disc $D$.
We use this disc $D$ to compute $\mu(K)$.

The loop corresponding to a double point arising from a crossing change consists in following the knot around, until we return to the double point.
The inverse loop is just the loop obtained by following the knot in the other direction.

Since $K$ intersects a 2-sphere $\geomi(K)$ times, one of the two loops will meet the two spheres at most $\geomi(K)/2$ times, and hence $|w(p)| \le \geomi(K)/2$.
Therefore, the degree of $\mu(K)$ is at most $\geomi(K)/2$.
\end{proof}

We now give a general computation of $\mu(K)$ for a very special family of knots.
All Whitehead and Bing doubling operations will be positively clasped and untwisted.
Fix a knot $T$, and let $W$ be its Whitehead double; let also $L = L_1\cup L_2$ and $J$ be the Bing and Whitehead double of $W$, respectively.
An observation that will be useful later is that $L$ is a symmetric link; i.e. there is an isotopy exchanging the two components.

Finally, let $K\subset \es$ be obtained by doing 0-surgery along $L_2$, as shown in Figure~\ref{f:bingwhitehead}.

\begin{figure}[h]
\labellist
\pinlabel $T$ at 184 70
\pinlabel $0$ at 233 266
\pinlabel $K$ at -25 138
\endlabellist
\includegraphics[width = 0.4\textwidth]{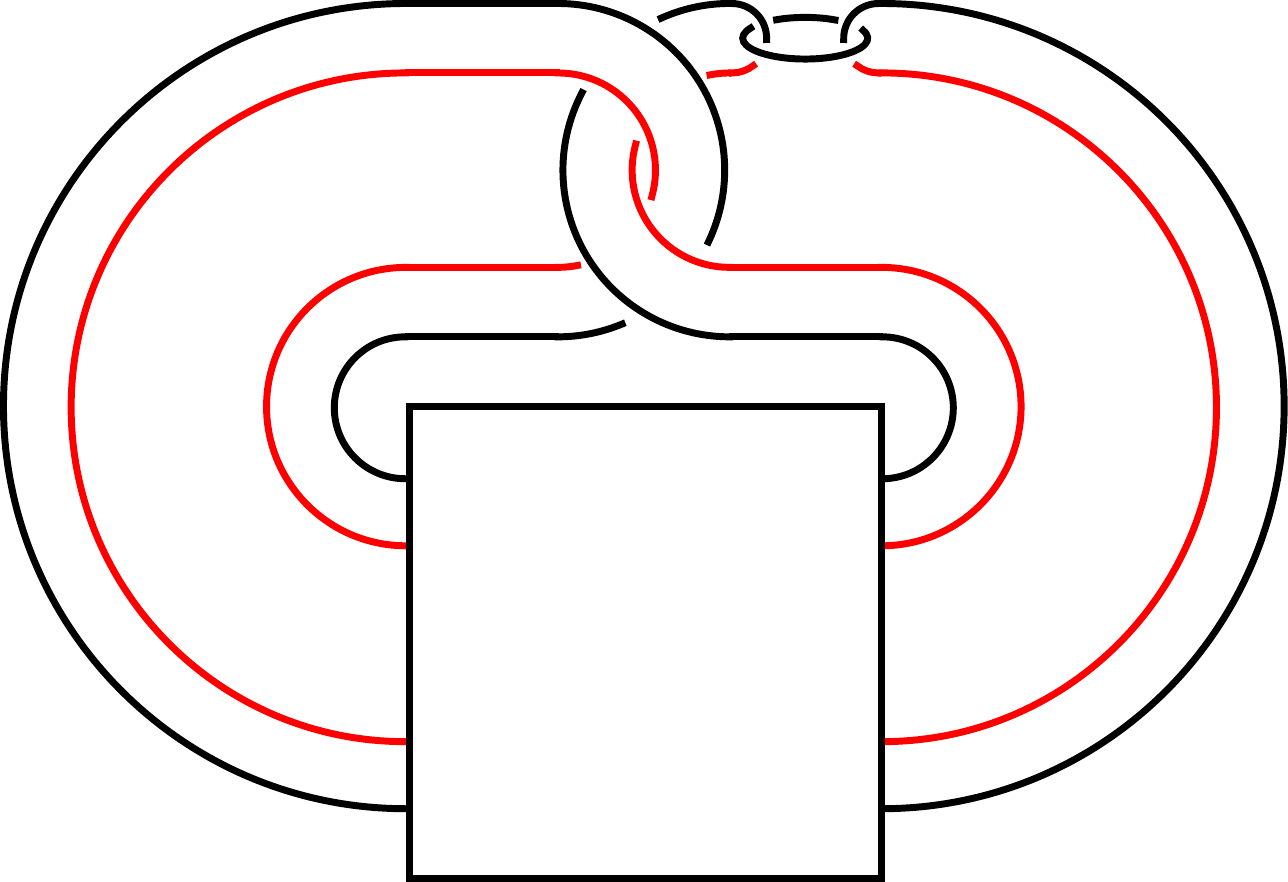}
\caption{The knot $K$. The two arcs $a_1$ and $a_2$ are in black and red respectively.}\label{f:bingwhitehead}
\end{figure}

The following lemma was suggested to us by JungHwan Park.

\begin{lemma}\label{l:schneiderman}
The knot $K$ described above has $\mu(K) = 0$.
\end{lemma}

\begin{proof}
The projection of $L_1$ is split into two arcs by the projection of $L_2$;
this divides $K$ into two arcs, $K = a_1 \cup a_2$ (these are displayed in red and black in Figure~\ref{f:bingwhitehead}).

Suppose that we have an unknotting sequence of $u$ crossing changes for $T$.
This corresponds to an unknotting sequence for $K$ comprising $16u$ crossing changes.
These crossing changes give an immersed disc in $\es\times I$, which we will use to compute $\mu(K)$;
we will show that each crossing change in $T$ corresponds to a trivial contribution from the corresponding sixteen crossing changes in $K$.

To this end, refer to Figure~\ref{f:multicrossing}.
A crossing change can be between a strand in $a_1$ and a strand in $a_2$, or between two strands on the same arc, say $a_1$.
In the latter case, we can connect the two lifts of the double point by an arc in $a_1$, and the corresponding loop is null-homotopic in $\es\times I$, so it does not contribute to $\mu(K)$.
Vice versa, if the two strands belong to two different arcs, when we connect them we cross a generating 2-sphere exactly once, and hence $w(p) = \pm 1$;
that is, each of the corresponding double points contributes with $\sigma(p)\cdot t$.

By counting directly around each crossing of $K$, as in Figure~\ref{f:multicrossing}, we see that there are four positive and four negative crossings, corresponding to four positive and four negative points in the immersed concordance.
Thus, the total contribution vanishes, and $\mu(K) = 0$.
\end{proof}

\begin{figure}
\labellist
\pinlabel $\leadsto$ at 195 90
\pinlabel $-$ at 323 129
\pinlabel $-$ at 323 45
\pinlabel $+$ at 306 112
\pinlabel $+$ at 306 63
\pinlabel $-$ at 371 129
\pinlabel $-$ at 371 45
\pinlabel $+$ at 386 112
\pinlabel $+$ at 386 63
\endlabellist
\includegraphics[scale = 0.6]{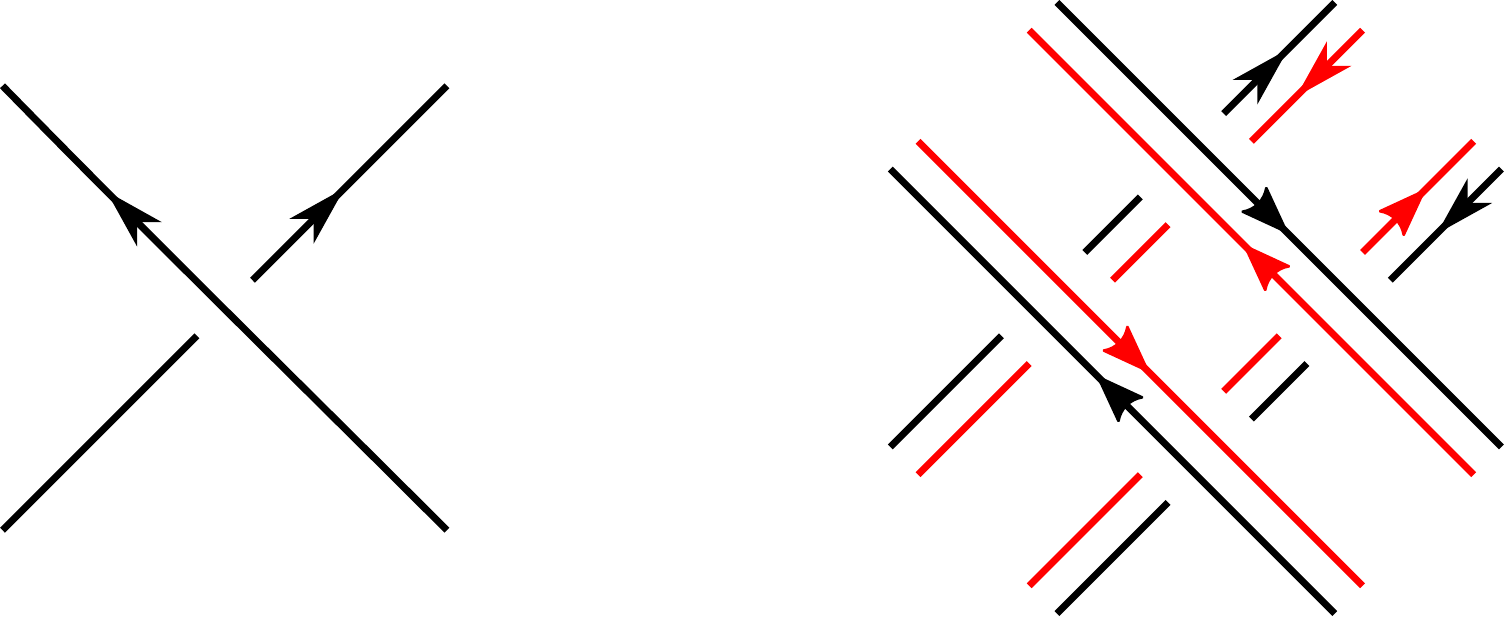}
\caption{The multiplication of crossings: on the left, a crossing in $T$; on the right, the corresponding sixteen crossings in $K$. We labelled the positive and negative crossings of `mixed types' (i.e. those where two strands of different colour meet).}\label{f:multicrossing}
\end{figure}
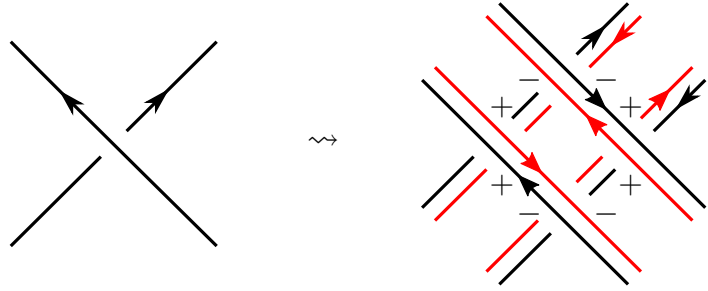

We now look at the 3-manifold $Y_K$ obtained as $+1$-surgery along $K$.
That is, $Y_K$ is obtained by doing $+1$-surgery on $L_1$ and $0$-surgery on $L_2$;
since $L$ is a symmetric link, we can blow down $L_1$, and the blowdown of $L_2$ will be $J$, the Whitehead double of $W$.
Therefore, we are in the assumption of Corollary~\ref{c:surgery}, and we want to compute $V_0(J)$ and $V_0(-J)$.

\begin{lemma}\label{l:bingwhitehead}
If the maximal Thurston--Bennequin number of $T$ is positive, then $V_0(J) = 1$.
\end{lemma}

\begin{proof}
By construction, $J$ has unknotting number 1 (by changing a crossing in the clasp);
more precisely, once can change a positive crossing into a negative one, and obtain an unknot.
Therefore, $V_0(-J) = 0$, and $V_0(J) \le 1$, by~\cite[Theorem 6.1]{BorodzikHedden};
to prove that $V_0(J) = 1$, we use the slice Bennequin inequality~\cite{Olga}:
namely, it is well-known that $W$ has a Legendrian representative with Thurston--Bennequin number 1, and that (untwisted, positively clasped) Whitehead doubling preserves this property~\cite{Rudolph};
hence, also $J$ has such a Legendrian representative, and this proves that $\tau(J) > 0$, which in turn proves that $V_0(J) > 0$~\cite[Proposition 7.7]{rasmussen2003floer}.
\end{proof}

Let now $T$ be any knot satisfying the assumption of Lemma~\ref{l:bingwhitehead}; for instance, $T$ can be chosen to be the right-handed trefoil.

\begin{prop}
The Schneiderman invariant $\mu(K)$ of $K$ vanishes, but $\geomi(K) = 2$.
\end{prop}

\begin{proof}
The Schneiderman invariant $\mu(K)$ vanishes, thanks to Lemma~\ref{l:schneiderman}.

Evidently, $\geomi(K) \le 2$.
The converse inequality follows from Corollary~\ref{c:surgery} and Lemma~\ref{l:bingwhitehead}:
indeed, doing $+1$-surgery along $K$ yields 3-manifold $Y$ that is obtained as $0$-surgery along $J$; by Lemma~\ref{l:bingwhitehead}, $V_0(J) = 1$, so Corollary~\ref{c:surgery} implies $\geomi(K) \ge 2$, as desired.
\end{proof}

Note that by combining~\cite[Corollary 1.3]{cimasoni2006slicing}, and the fact that Whitehead doubles are always topologically slice by a result of Freedman~\cite{FreedmanWhitehead}, we obtain that the knot $K$ is topologically slice in $\es = \partial S^2 \times D^2$. 

In particular this implies that the bound~\eqref{e:geomlowerboundN} detects the difference between topologically and smoothly slice.

\subsection{Sharp, arbitrarily large bounds}

As promised, we construct an infinite family of knots $K_n$, indexed by positive integers;
the knots will be given by the diagram in Figure~\ref{f:ex_diagram}.

\begin{figure}[h]
\labellist
\pinlabel $K_n$ at 120 -1
\pinlabel {{\color{red}$0$}} at 312 158
\pinlabel $1$ at 34 76
\pinlabel $2n+1$ at 373 116
\pinlabel $2n+1$ at 373 44
\endlabellist
\includegraphics[scale = 0.75]{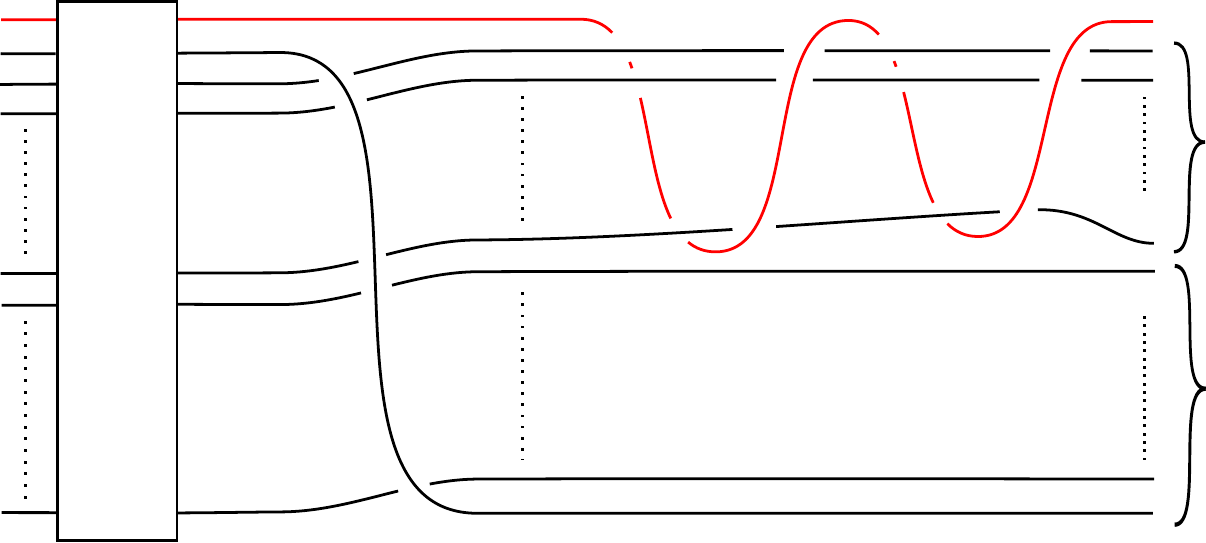}
\caption{The knot $K_n$. The figure represents a $(4n+3)$-braid, the box is a full twist. The closure of the braid has two components, one of which (in red, at the top of the figure) is an unknot, along which we do 0-surgery; the other component, $K_n$, has linking number $0$ with the first component, and hence represents a null-homologous knot in $\es$.}\label{f:ex_diagram}
\end{figure}
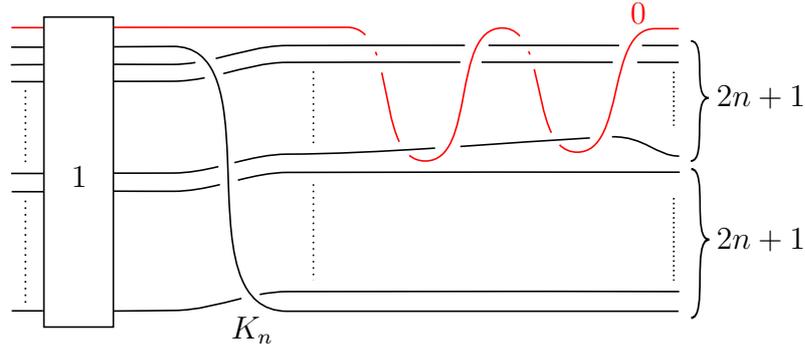

\begin{prop}\label{p:esempi}
The knot in $K_n$ has $\geomi(K_n) = 4n+2$.
\end{prop}

Note that the obstruction of Theorem~\ref{t:geomlowerbound} or Theorem~\ref{t:main} cannot see the difference between $\geomi(K_n) = 4n+2$ or $\geomi(K_n) = 4n+4$;
that is to say, if $\geomi(K_n)$ is in fact equal to $4n+4$, this cannot be detected by our results, which can only guarantee $\geomi(K_n) \ge 4n+2$.
As a consequence, this family consists of pairwise mutually non smoothly almost-concordant knots in the trivial free homotopy class of $\es$; the existence of such a family of knots in arbitrary $3$-manifold was enstablished by Friedl--Nagel--Orson--Powell~\cite[Theorem 1.6]{friedl2016satellites}, and later on by Yildiz~\cite{yildiz2017note}.

The strategy of proof is quite straightforward: we need to exhibit a 2-sphere representing the generator of $H_2(\es)$ that meets $K_n$ in $4n+2$ points, and we want to apply Theorem~\ref{t:geomlowerbound} to some \spinc structure on some positive surgery along $K_n$.
The 2-sphere is in fact easy to spot, as shown in Figure~\ref{f:disco}.

\begin{figure}
\labellist
\pinlabel $K_n$ at 91 6
\pinlabel {{\color{red}$0$}} at 330 157
\pinlabel $1$ at 23 79
\endlabellist
\includegraphics[scale = 0.78]{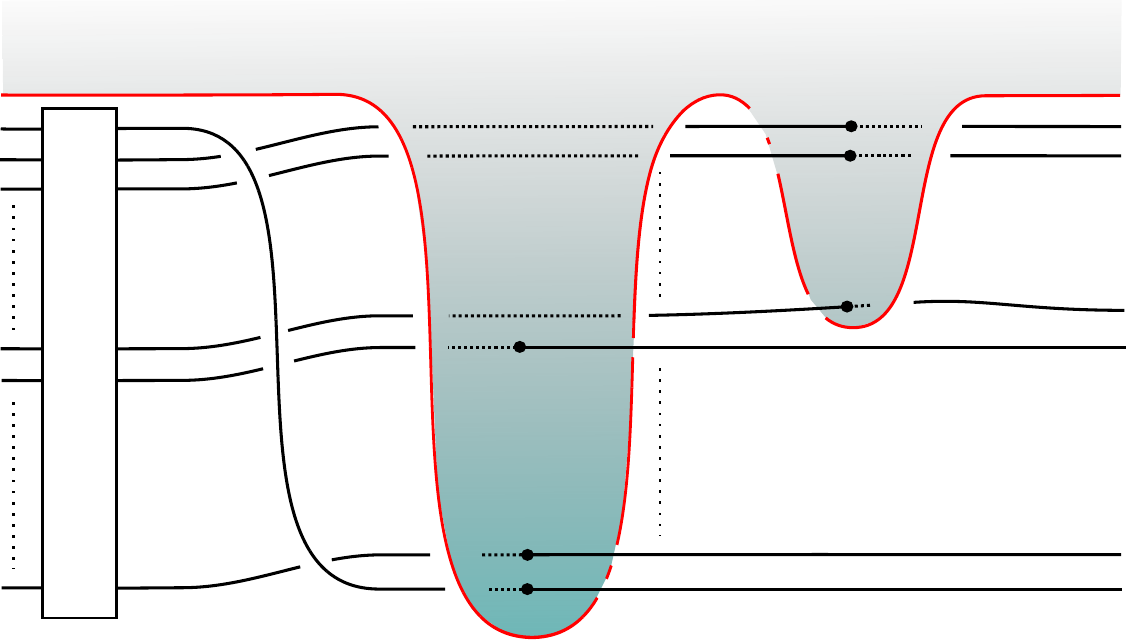}
\caption{This is the knot $K_n$, where we singled out the unknotted strand from the full twist. The disc is shaded, and the points of intersections are marked with a bullet.}\label{f:disco}
\end{figure}

On the other hand, computing correction terms is not an easy task.
For the manifold at hand, that is $m$-surgery on $K_n$ for some $m$, we proceed as follows.
We start with the Kirby diagram for $K_n$ of Figure~\ref{f:disco}, which comprises a 0-framed unknot $\bigcirc$ and a torus knot $J = T_{4n+2,4n+3}$, and we observe that this manifold fits into a triad, corresponding to doing surgery along $\bigcirc$ with coefficients $-1$, $0$, and $\infty$.
Call $(S^3,J')$ the knot obtained by doing $-1$-surgery along $\bigcirc \subset (S^3,J)$.

Doing $\infty$-surgery along $\bigcirc$ gives back $S^3$, together with the knot $J\subset S^3$.
The 3-manifold $S^3_m(J)$ is an L-space when $m\ge (4n+1)(4n+2)-1$, and its correction terms are well understood in terms of the semigroup generated by $4n+2$ and $4n+3$ in the non-negative integers (see~\cite{borodziklivingston}).

The following two lemmas are the key topological observation underpinning the proof of Proposition~\ref{p:esempi}.

\begin{lemma}\label{l:T3737}
The knot $J'$ is $T_{2n+1,4n+3}\# T_{2n+1,4n+3}$.
\end{lemma}

\begin{lemma}\label{l:T3737surgery}
The $3$-manifold $S^3_{(4n+2)(4n+3)}(J')$ is $M(-2;\frac{2n}{2n+1},\frac{2n}{2n+1},\frac2{4n+3},\frac2{4n+3})$, a Seifert fibred space over $S^2$ with four singular fibres.
\end{lemma}

That is, if we choose $m = (4n+2)(4n+3)$, $S^3_m(J)$ is an L-space (in fact, it is $L(4n+3,1)\#L(4n+2,4n+1)$) and $S^3_m(J')$ is a Seifert fibred space with Euler number
\[
-2+2\left(1-\frac1{2n+1}\right) + 2\cdot \frac2{4n+3} = 2\left(\frac2{4n+3}-\frac1{2n+1}\right) < 0.
\]
From~\cite[Corollary 1.4]{OSz-plumbing} we deduce that $\HFo(S^3_{(4n+2)(4n+3)}(J')) = 0$.

We defer the proof of the lemmas above, and we patch the argument together to prove Proposition~\ref{p:esempi} first.

\begin{proof}[Proof of Proposition~\ref{p:esempi}]
There is an obvious 2-sphere intersecting geometrically $K_n$ exactly $4n+2$ times, obtained by capping off the 2-disc shaded in Figure~\ref{f:disco}.
Therefore, $\geomi(K_n) \le 4n+2$.
We now set out to prove the opposite inequality.

Let $m = (4n+2)(4n+3)$, and let us look at the surgery exact triangle for the triad $Y_\infty = S^3_m(J)$, $Y_{-1} = S^3_m(J')$, and $Y_0 = (\es)_m(K)$ above.
\[
\xymatrix{
\HF(Y_\infty)[t,t\inv] \ar[rr]^F & & \HF(Y_{-1})[t,t\inv]\ar[dl]^G\\
 & \tHF(Y_0)\ar[ul]^H
}
\]
Since the cobordism $Y_\infty \leadsto Y_{-1}$ inducing $F$ is obtained by attaching a ($-1$)-framed 2-handle along a null-homologous knot in $Y_\infty$, $F$ maps $\HF(Y_\infty,\ft_i)[t,t\inv]$ to $\HF(Y_{-1},\ft_i)[t,t\inv]$ for each $i$.
The same holds for the cobordism inducing $G$.

For each $i$, the map $F_i: \HF(Y_\infty,\ft_i)[t,t\inv] \to \HF(Y_{-1},\ft_i)[t,t\inv]$ is, up to higher order terms in $U$, $U^k\cdot (1-t)$:
indeed, the tower in $\tHF(Y_0)$ is isomorphic to $\F[U,U\inv]/U\cdot\F[U]$, with all elements of $H_2(Y_0)$ acting as the identity on it, so $H$ vanishes on the towers, and $G$ surjects onto them.

Since $Y_\infty$ is an L-space, the tower in $\HF(Y_{-1})[t,t\inv]/(1-t)$ maps injectively into $\tHF(Y_0,\ft_i)$, so that $G$ is surjective on the tower.
In particular, no tower in $\HF(Y_\infty,\ft_i)$ is in the image of $H$, so that each tower in $\HF(Y_\infty,\ft_i)$ maps injectively into $\HF(Y_{-1},\ft_i)$.
Computing the gradings of the maps involved, this proves that
\[
\dtw(Y_0,\ft_0) = - 2V_0(J') + \frac{m-3}{4}.
\]

Let us now look at the triad $-Y_\infty$, $-Y_0$, $-Y_{-1}$.
\[
\xymatrix{
\HF(-Y_{-1})[t,t\inv] \ar[rr]^{F'} & & \HF(-Y_\infty)[t,t\inv]\ar[dl]^{G'}\\
 & \tHF(-Y_0)\ar[ul]^{H'}
}
\]
The key observation that makes the same argument run is that $\HFred(-Y_{-1})$ is now supported in odd degrees, while the map $F'$ is a sum of maps of even degree;
it follows that each tower $\HF(-Y_\infty,\ft_i)[t,t\inv]/(1-t)$ is mapped isomorphically into $\tHF(-Y_0,\ft_i)$, and therefore
\[
\dtw(-Y_0,\ft_0) = 2V_0(J) - \frac{m+1}4.
\]

Applying Theorem~\ref{t:geomlowerbound}, Lemmas~\ref{l:dT3737} and~\ref{l:V0}, we obtain that
\begin{align*}
\dtw(Y_0) + \dtw(-Y_0) + 1&= 2V_0(J) - 2V_0(J') = 2V_0(T_{4n+2,4n+3}) - 2V_0(T_{2n+1,8n+5}) \\
&= 2(n+1)(2n+1) - 4n(n+1) = 2n+2 \le 2\left\lceil \frac{\geomi(K_n)}4\right\rceil,
\end{align*}
from which $\geomi(K_n) \ge 4n+2$ follows.
\end{proof}

We end this section with the proofs of the two lemmas above.

\begin{figure}
\labellist
\pinlabel (a) at -20 820
\pinlabel (b) at 780 820
\pinlabel (c) at -20 570
\pinlabel (d) at 780 570
\pinlabel (e) at -20 320
\pinlabel (f) at 780 320
\pinlabel (g) at 190 75
\pinlabel {\tiny{\color{blue}$-1$}} at 24 715
\pinlabel {\tiny{\color{red}$-2$}} at 300 899
\pinlabel ${2n+1}$ at 362 855
\pinlabel $2n+1$ at 362 778
\pinlabel {\tiny{\color{blue}$-1$}} at 474 715
\pinlabel {\tiny{\color{red}$0$}} at 749 899
\pinlabel {\tiny{\color{green}$1$}} at 715 913
\pinlabel {\tiny{\color{green}$1$}} at 640 913
\pinlabel {\tiny{\color{purple}$-1$}} at 687 848
\pinlabel {\tiny{\color{purple}$-1$}} at 585 848
\pinlabel {\tiny{\color{blue}-$1$}} at 50 470
\pinlabel {\tiny{\color{red}$0$}} at 278 656
\pinlabel {\tiny{\color{green}$1$}} at 198 669
\pinlabel {\tiny{\color{green}$2$}} at 165 604
\pinlabel {\tiny{\color{purple}$-1$}} at 222 604
\pinlabel {\tiny{\color{purple}$-2$}} at 282 604
\pinlabel {\tiny{\color{blue}$0$}} at 593 462
\pinlabel {\tiny{\color{red}$0$}} at 597 676
\pinlabel {\tiny{\color{green}$1$}} at 560 605
\pinlabel {\tiny{\color{green}$2$}} at 627 562
\pinlabel {\tiny{\color{purple}$-1$}} at 663 564
\pinlabel {\tiny{\color{purple}$-2$}} at 692 654
\pinlabel {\tiny{\color{blue}$0$}} at 146 229
\pinlabel {\tiny{\color{green}$2$}} at 184 329
\pinlabel {\tiny{\color{purple}$-1$}} at 220 330
\pinlabel {\tiny{\color{purple}$-2$}} at 247 419
\pinlabel {\tiny{\color{blue}$-1$}} at 593 230
\pinlabel {\tiny{\color{green}$-2$}} at 631 329
\pinlabel {\tiny{\color{purple}$-1$}} at 668 333
\pinlabel {\tiny{\color{purple}$-2$}} at 692 419
\pinlabel $2$ at 377 35
\pinlabel $2$ at 442 127
\endlabellist
\includegraphics[width = 0.9\textwidth]{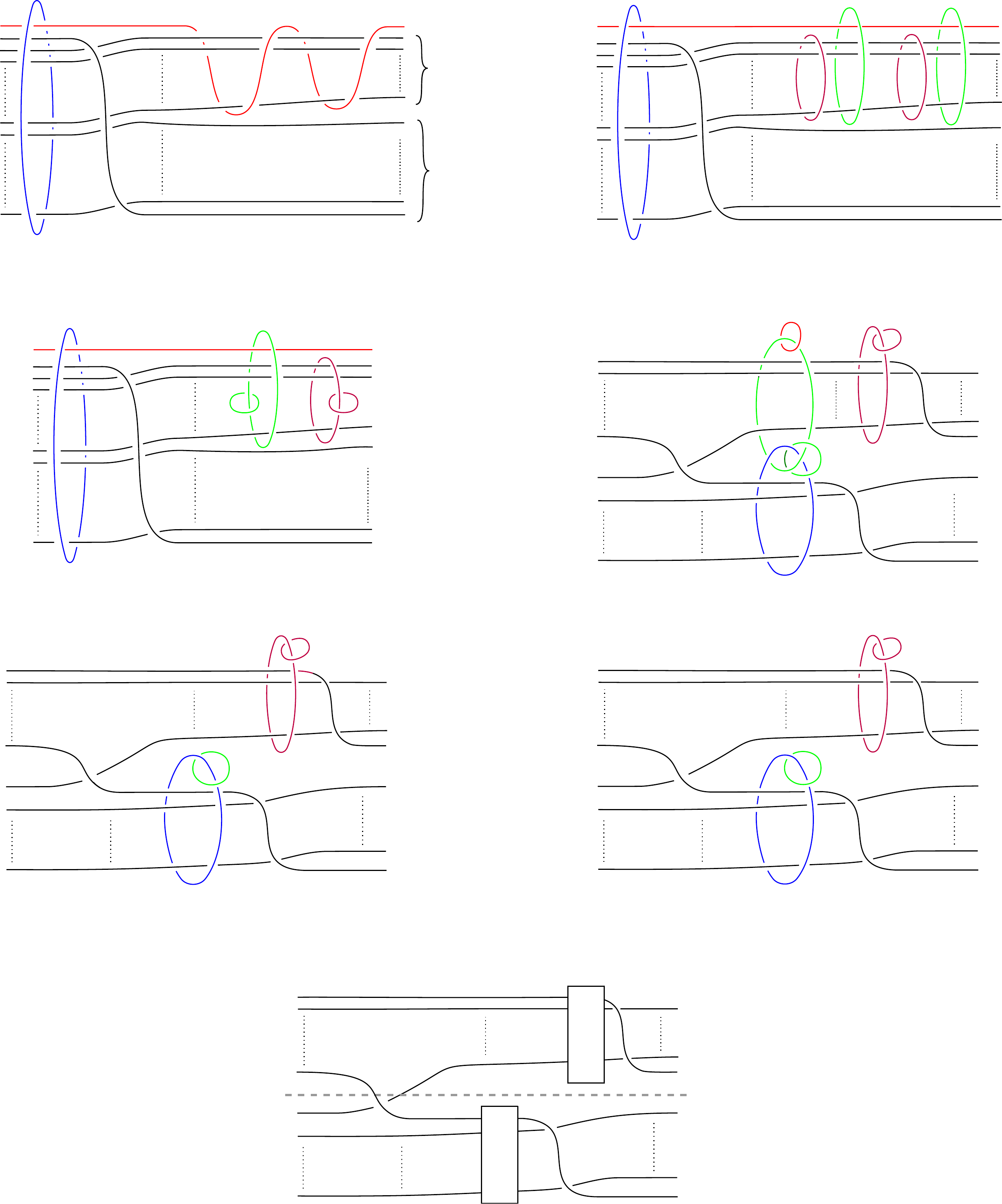}
\caption{The proof of Lemma~\ref{l:T3737}.}\label{f:T3737}
\end{figure}
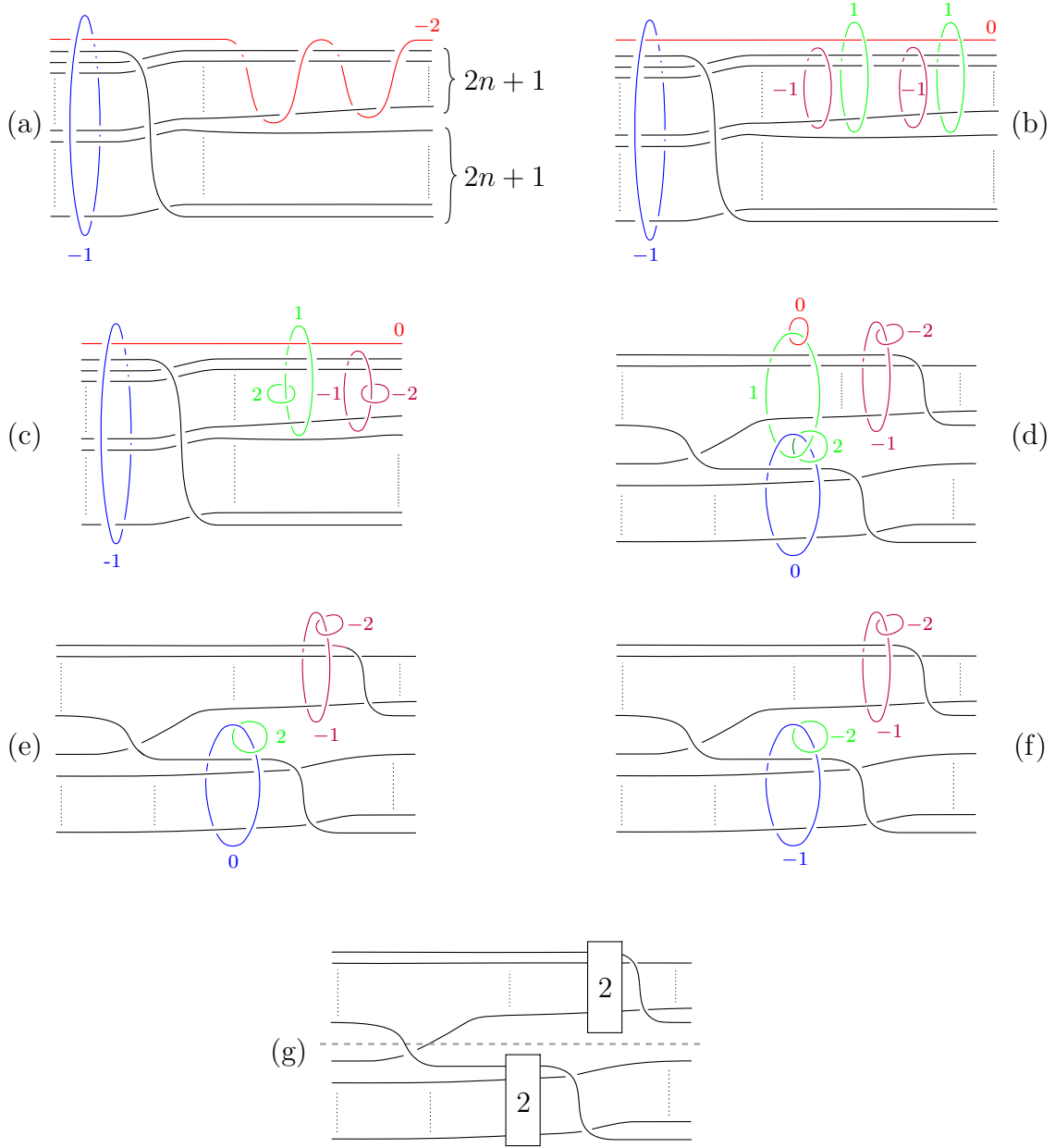

\begin{proof}[Proof of Lemma~\ref{l:T3737}]
We give a diagrammatic proof, following Figure~\ref{f:T3737}. From top to bottom:
\begin{itemize}
\item[(a)] This is obtained from $(S^3,\bigcirc \cup J)$ by blowing up along the blue curve.
\item[(b)] This is obtained from (a) by blowing up along the purple curves and blowing up negatively along the green curves.
\item[(c)] This is obtained from (b) by sliding one of the green curves over the other, and one of the purple curves over the other.
\item[(d)] This is obtained from (c) by sliding the blue curve over the $+1$-framed green curve.
\item[(e)] This is obtained from (d) by doing a slam dunk of the $0$-framed red curve; this amounts to cancelling both the red component and the $+1$-framed green component.
\item[(f)] This is obtained from (e) by doing a Rolfsen twist along the green curve.
\item[(g)] This is obtained from (f) by blowing down the blue curve and the $-1$-framed purple curve, and then the green curve and the remaining purple curve.
\end{itemize}
Note that (g) displays exactly the connected sum $T_{2n+1,4n+3}\#T_{2n+1,4n+3}$: the dashed line exhibits the 2-sphere giving the connected sum decomposition.
\end{proof}

While it is not necessary for the proof, as a litmus test, we also check that the framing of $J'$ is preserved in the sequence of moves above.
Indeed, following each of the steps, the framing decreases by $(4n+2)^2$ in the first step, and stays constant until the last, when it increases again by $4\cdot(2n+1)^2$.

\begin{figure}[t]
\labellist
\pinlabel ($a^\prime$) at -15 755
\pinlabel $(4n+2)(4n+3)$ at 52 633
\pinlabel $2$ at 17 717
\pinlabel $2$ at 169 697
\pinlabel ($b^\prime$) at 485 755
\pinlabel $4n+2$ at 330 633
\pinlabel $-1$ at 271 737
\pinlabel $-1$ at 477 711
\pinlabel ($c^\prime$) at 32 551
\pinlabel $4n+2$ at 222 605
\pinlabel $-1$ at 130 605
\pinlabel $-1$ at 313 605
\pinlabel $2n+1$ at 176 551
\pinlabel $2n+1$ at 265 551
\pinlabel ($d^\prime$) at 32 416
\pinlabel $0$ at 225 470
\pinlabel $-2(n+1)$ at 105 470
\pinlabel $-2(n+1)$ at 336 470
\pinlabel $-1$ at 170 450
\pinlabel $-1$ at 283 450
\pinlabel ($e^\prime$) at 32 190
\pinlabel $-2(n+1)$ at 104 266
\pinlabel $-2(n+1)$ at 374 222
\pinlabel $-1$ at 258 251
\pinlabel $-1$ at 290 175
\endlabellist
\includegraphics[width = 0.75\textwidth]{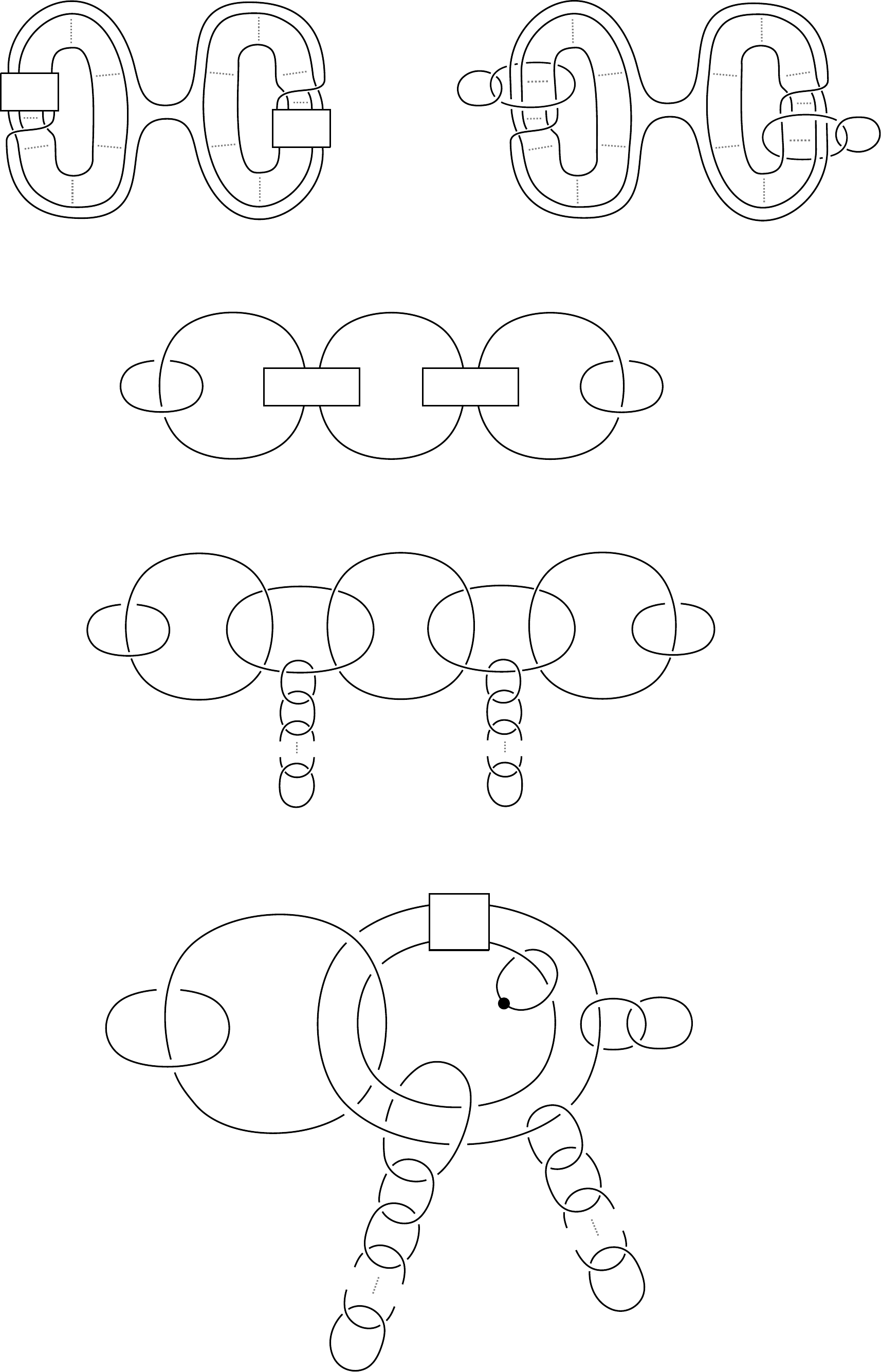}
\caption{The proof of Lemma~\ref{l:T3737surgery}. We abide by the convention that we do not label $-2$-framed components. In ($b^\prime$) the two ``outer'' unknots are $-2$-framed; in ($d^\prime$) and ($e^\prime$), each of the two chains of (unlabelled, hence $-2$-framed) unknots has length $2n$.}\label{f:seifert}
\end{figure}
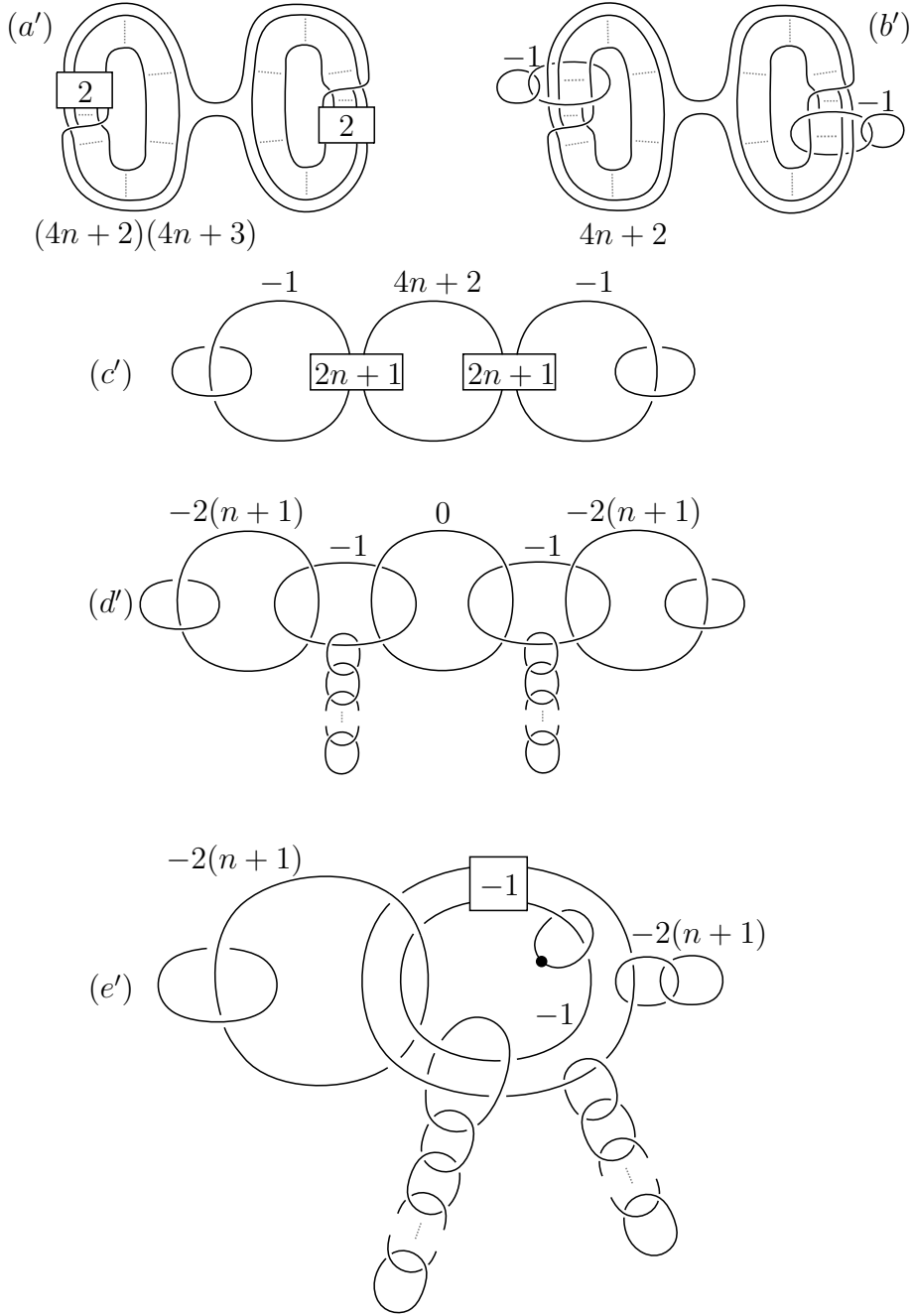
\begin{proof}[Proof of Lemma~\ref{l:T3737surgery}]
We start from (f) in Figure~\ref{f:T3737}; note that the framing on (the component corresponding to) $J'$ is now $(4n+2)(4n+3)-4(2n+1)^2 = 4n+2$.
We then refer to Figure~\ref{f:seifert}.
\begin{itemize}
\item[($a^\prime$)] This is just obtained from (g) in Figure~\ref{f:T3737} by an isotopy; the blue component correspond to $J'$, and has framing $4n+2$.
\item[($b^\prime$)] This is essentially (f) in Figure~\ref{f:T3737}.
\item[($c^\prime$)] Is obtained by an isotopy from ($b^\prime$).
\item[($d^\prime$)] This is obtained from ($c^\prime$) by blowing up the $(2n+1)+(2n+1)$ twists, and sliding each new $-1$-framed unknot over the next, as done to go from step (b) to (c) in the proof of Lemma~\ref{l:T3737}.
\item[($e^\prime$)] This is obtained by sliding one of the $-1$-framed curves over the other, and by doing 0-dot surgery on the 0-framed component.
\end{itemize}
We can now cancel the 1-handle in Figure~\ref{f:T3737} with the $-1$-framed knot, and therefore obtain a presentation of $S^3_{(4n+2)(4n+3)}(J')$ as a Seifert fibred space over $S^2$ with four singular fibres.
The corresponding Seifert invariants are easily computed from the (negative) continued fraction expansions $[2,\dots,2]^- = \frac{2n+1}{2n}$ and $[2n+2,2]^- = \frac{4n+3}2$.
\end{proof}

We note that that ($c^\prime$) in Figure~\ref{f:seifert} gives a plumbed presentation of $S^3_{(4n+2)(4n+3)}(J')$; the final presentation can also be obtained in the plumbing language by doing a 0-absorption move~\cite{Neumann}.

\appendix
\section{The $0$-shake-slice genus (with Adam Levine)}

The goal of this appendix is to prove Theorem~\ref{t:shake-bound}.
The techniques are similar to the ones employed in the rest of the paper.

Let us start by setting up some notation.
If $K$ is a knot in $S^3$, denote with $X_K$ the trace of the 0-surgery along $K$, which is $B^4$ with a 2-handle attached along $K$ with framing 0;
we write $Y_K = S^3_0(K)$ for the boundary of $X_K$.
We also denote with $-K$ the mirror of $K$, with its orientation reversed.
(However, the orientation will not play any role.)

Recall that the $0$-shake-slice genus $\gsh(K)$ of $K$ is the minimal genus $g(F)$ of a smoothly embedded surface $F$ representing a generator of $H_2(X_K)$.

In the proof, we will let $F$ be a surface whose fundamental class generates $H_2(X_K)$. Let $W = -(X_K \setminus N)$, where $N$ is an open regular neighbourhood of $F$; notice that since $F\cdot F = 0$, $N \cong F\times D^2$ and $\de W = -Y_K \sqcup S$, where $S \cong F\times S^1$.
We will view $W$ as a cobordism from $Y_K$ to $S$.

We want to apply Theorem~\ref{t:dtwcobordism}; the following lemma is the analogue of Lemma~\ref{l:injectiveH1} above.

\begin{lemma}
The inclusion $Y_K \hookrightarrow W$ induces an isomorphism $H_1(Y_K) \to H_1(W)$.
\end{lemma}

\begin{proof}
The Mayer--Vietoris long exact sequence for the decomposition $X_K = W \cup N$ reads:
\[
\xymatrix{
H_2(W)\oplus H_2(N) \ar[r]^-{\alpha} & H_2(X_K)\ar[r] & H_1(S)\ar[r]^-\beta & H_1(W)\oplus H_1(N)\ar[r] & H_1(X_K).
}
\]
Since $[F]$ is a generator of $H_2(X_K)$ by assumption, the map $H_2(N) \to H_2(X_K)$ induced by the inclusion is surjective, and therefore so is $\alpha$.
This implies that the map $\beta$ is injective, and hence, since $H_1(X_K) = 0$, an isomorphism.

Recall now that $H_1(S) = H_1(F) \oplus \Z[f]$, where $f$ is the $S^1$-fibre.
The map $\beta$ is an isomorphism onto $H_1(N)$ when restricted to the summand $H_1(F)$ of $H_1(S)$; moreover, the fibre $f$ is in the kernel of the inclusion $H_1(S)\to H_1(F)$, and hence $\beta$ maps $[f]$ to a generator of $H_1(W)$.

By construction, though, $[f]\in H_1(S)$ is homologous in $W$ to a generator of $H_1(Y_K)$, and hence the inclusion $Y_K\hookrightarrow W$ induces an isomorphism on $H_1$.
\end{proof}

As we did above, we will omit the \spinc structure from the notation, when there is a unique torsion \spinc structure.

\begin{proof}[Proof of Theorem~\ref{t:shake-bound}]
By the lemma above, the inclusion $Y_K\hookrightarrow W$ induces an injection $H_1(Y_K) \to H_1(W)$.
Moreover, since the intersection form of $X_K$ is $(0)$, $W$ is a negative semi-definite cobordism.

The assumptions to apply Theorem~\ref{t:dtwcobordism} are satisfied, and we can write the inequality
\[
4\dtw(Y_K) + 2b_1(Y_K) + c_1(\fs)^2 + b^-(W) \le 4\dtw(S) + 2b_1(S).
\]
By Proposition~\ref{p:dtwcomputations}, $\dtw(S) = \frac12(-1)^{g+1}$.
Since $W$ has trivial intersection form, $c_1(\fs)^2 = 0$, and therefore
 \[
4\dtw(Y_K) + 2 \le 8\left\lceil \frac g2 \right\rceil,
\]
from which the statement follows.
\end{proof}

We can also apply the same theorem to $-K$;
this, and the fact that $\dtw(Y_K) = 2V_0(-K)-\frac12$ (Proposition~\ref{p:0surgery}), allows us to recast the statement of Theorem~\ref{t:shake-bound} as follows:
\[
\gsh(K) \ge 2\max\{V_0(K), V_0(-K)\} - 1.
\]

\bibliographystyle{amsplain}
\bibliography{nodificazione}

\providecommand{\bysame}{\leavevmode\hbox to3em{\hrulefill}\thinspace}
\providecommand{\MR}{\relax\ifhmode\unskip\space\fi MR }
\providecommand{\MRhref}[2]{%
  \href{http://www.ams.org/mathscinet-getitem?mr=#1}{#2}
}
\providecommand{\href}[2]{#2}
\begin{thebibliography}{10}

\bibitem{behrens2015heegaard}
Stefan Behrens and Marco Golla, \emph{Heegaard {F}loer correction terms, with a
  twist}, Quantum Topol. \textbf{9} (2018), no.~1, 1--37.

\bibitem{bodnar2016lattice}
J{\'o}zsef Bodn{\'a}r and Andr{\'a}s N{\'e}methi, \emph{Lattice cohomology and
  rational cuspidal curves}, Math. Res. Lett. \textbf{23} (2016), no.~2,
  339--375.

\bibitem{BorodzikHedden}
Maciej Borodzik and Matthew Hedden, \emph{The {U}psilon function of {L}--space
  knots is a {L}egendre transform}, Math. Proc. Camb. Philos. Soc. \textbf{164}
  (2018), no.~3, 401--411.

\bibitem{borodziklivingston}
Maciej Borodzik and Charles Livingston, \emph{Heegaard {F}loer homology and
  rational cuspidal curves}, Forum Math. Sigma \textbf{2} (2014), e28, 23.

\bibitem{celoria2018concordances}
Daniele Celoria, \emph{On concordances in 3-manifolds}, J. Topol. \textbf{11}
  (2018), no.~1, 180--200.

\bibitem{cimasoni2006slicing}
David Cimasoni, \emph{Slicing {B}ing doubles}, Algebr. Geom. Topol. \textbf{6}
  (2006), no.~5, 2395--2415.

\bibitem{davis2017concordance}
Christopher~W. Davis, Matthias Nagel, JungHwan Park, and Arunima Ray,
  \emph{Concordance of knots in {$S^1\times S^2$}}, to appear in J. London
  Math. Soc., 2017.

\bibitem{donald2012concordance}
Andrew Donald and Brendan Owens, \emph{Concordance groups of links}, Algebr.
  Geom. Topol. \textbf{12} (2012), no.~4, 2069--2093.

\bibitem{FreedmanWhitehead}
Michael~H. Freedman, \emph{A new technique for the link slice problem}, Invent.
  Math. \textbf{80} (1985), no.~3, 453--465.

\bibitem{friedl2016satellites}
Stefan Friedl, Matthias Nagel, Patrick Orson, and Mark Powell, \emph{Satellites
  and concordance of knots in 3-manifolds}, arXiv preprint arXiv:1611.09114
  (2016).

\bibitem{gabai1987foliations}
David Gabai, \emph{Foliations and the topology of 3-manifolds. {II}}, J. Diff.
  Geom. \textbf{26} (1987), no.~3, 461--478.

\bibitem{heddenkuzbary}
Matthew Hedden and Miriam Kuzbary, in preparation.

\bibitem{hosokawa1967concept}
Fujitsugu Hosokawa, \emph{A concept of cobordism between links}, Ann. Math.
  (1967), 362--373.

\bibitem{kirbylist}
Rob Kirby, \emph{Problems in low dimensional manifold theory}, Algebraic and
  geometric topology ({P}roc. {S}ympos. {P}ure {M}ath., {S}tanford {U}niv.,
  {S}tanford, {C}alif., 1976), {P}art 2, Proc. Sympos. Pure Math., XXXII, Amer.
  Math. Soc., Providence, R.I., 1978, pp.~273--312.

\bibitem{LevineRuberman}
Adam~Simon Levine and Daniel Ruberman, \emph{Heegaard {F}loer invariants in
  codimension one}, to appear in Trans. Amer. Math. Soc., 2016.

\bibitem{levine2015nonorientable}
Adam~Simon Levine, Daniel Ruberman, and Sa{\v{s}}o Strle, \emph{Nonorientable
  surfaces in homology cobordisms}, Geom. Topol. \textbf{19} (2015), no.~1,
  439--494.

\bibitem{wrappinglivingston}
Charles Livingston, \emph{Mazur manifolds and wrapping numbers of knots in
  ${S}^2 \times {S}^1$}, Houston J. Math. \textbf{11} (1985), no.~4, 523--533.

\bibitem{Neumann}
Walter~D. Neumann, \emph{A calculus for plumbing applied to the topology of
  complex surface singularities and degenerating complex curves}, Trans. Amer.
  Math. Soc. \textbf{268} (1981), no.~2, 299--344.

\bibitem{ni2015correction}
Yi~Ni and Zhongtao Wu, \emph{Correction terms, {$\mathbb{Z}_2$}-{T}hurston
  norm, and triangulations}, Topology and its Applications \textbf{194} (2015),
  409--426.

\bibitem{NiWu}
\bysame, \emph{Cosmetic surgeries on knots in {$S^3$}}, J. Reine Angew. Math.
  \textbf{2015} (2015), no.~706, 1--17.

\bibitem{ozsvath2008holomorphic}
Peter Ozsv{\'a}th and Zolt{\'a}n Szab{\'o}, \emph{Holomorphic disks, link
  invariants and the multi-variable {A}lexander polynomial}, Algebr. Geom.
  Topol. \textbf{8} (2008), no.~2, 615--692.

\bibitem{OSz-integralsurgeries}
\bysame, \emph{Knot {F}loer homology and integer surgeries}, Algebr. Geom.
  Topol. \textbf{8} (2008), no.~1, 101--153.

\bibitem{ozsvath2008link}
\bysame, \emph{Link {F}loer homology and the {T}hurston norm}, J. Amer. Math.
  Soc. \textbf{21} (2008), no.~3, 671--709.

\bibitem{OSz-absolutely}
Peter~S. Ozsv{\'a}th and Zolt{\'a}n Szab{\'o}, \emph{Absolutely graded {F}loer
  homologies and intersection forms for four-manifolds with boundary}, Adv.
  Math. \textbf{173} (2003), no.~2, 179--261.

\bibitem{OSz-plumbing}
\bysame, \emph{On the {F}loer homology of plumbed three-manifolds}, Geom.
  Topol. \textbf{7} (2003), 185--224.

\bibitem{ozsvath2004holomorphicknot}
\bysame, \emph{Holomorphic disks and knot invariants}, Adv. Math. \textbf{186}
  (2004), no.~1, 58--116.

\bibitem{ozsvath2004holomorphicpropr}
\bysame, \emph{Holomorphic disks and three-manifold invariants: properties and
  applications}, Ann. Math. \textbf{159} (2004), no.~3, 1159--1245.

\bibitem{ozsvath2004holomorphic}
\bysame, \emph{Holomorphic disks and topological invariants for closed
  three-manifolds}, Ann. Math. \textbf{159} (2004), no.~3, 1027--1158.

\bibitem{piccirillo}
Lisa Piccirillo, \emph{Shake genus and slice genus}, arXiv preprint
  \href{arXiv.org/abs/1803.09834}{arXiv:1803.09834}, 2018.

\bibitem{Olga}
Olga Plamenevskaya, \emph{Bounds for the {T}hurston-{B}ennequin number from
  {F}loer homology}, Algebr. Geom. Topol. \textbf{4} (2004), no.~1, 399--406.

\bibitem{rasmussen2003floer}
Jacob~A. Rasmussen, \emph{Floer homology and knot complements}, Ph.D. thesis,
  Harvard University, 2003.

\bibitem{Rasmussen-GodaTeragaito}
\bysame, \emph{Lens space surgeries and a conjecture of {G}oda and
  {T}eragaito}, Geom. Topol. \textbf{8} (2004), no.~3, 1013--1031.

\bibitem{Rudolph}
Lee Rudolph, \emph{An obstruction to sliceness via contact geometry and
  ``classical'' gauge theory}, Invent. Math. \textbf{119} (1995), no.~1,
  155--163.

\bibitem{schneiderman2003algebraic}
Rob Schneiderman, \emph{Algebraic linking numbers of knots in 3--manifolds},
  Algebr. Geom. Topol. \textbf{3} (2003), no.~2, 921--968.

\bibitem{thurston1986norm}
William~P. Thurston, \emph{A norm for the homology of 3-manifolds}, Mem. Amer.
  Math. Soc. \textbf{339} (1986), 99--130.

\bibitem{yildiz2017note}
Eylem~Zeliha Yildiz, \emph{A note on knot concordance}, arXiv preprint
  \href{arXiv.org/abs/1707.01650}{arXiv:1707.01650}, 2017.

\end{thebibliography}
\end{document}